\documentclass[12pt, a4paper]{amsart}
\usepackage{amssymb,fullpage}
\usepackage{amsmath,esint}
\usepackage[colorlinks, allcolors=RedViolet]{hyperref}
\usepackage{colortbl}
\usepackage[dvipsnames]{xcolor}

\newtheorem{theorem}[equation]{Theorem}
\newtheorem{lemma}[equation]{Lemma}
\newtheorem{proposition}[equation]{Proposition}

\theoremstyle{definition}
\newtheorem{definition}[equation]{Definition}

\theoremstyle{remark}
\newtheorem{remark}[equation]{Remark}

\numberwithin{equation}{section}

\newcommand{ \mr }{ \mathbb{R} }
\newcommand{ \R }{ \mathbb{R} }

\newcommand{\tphi}{{\tilde\phi}}

\newcommand{\loc}{{\operatorname{loc}}}
\DeclareMathOperator*{\esssup}{\operatorname{ess\,sup}}
\renewcommand{\epsilon}{\varepsilon}
\renewcommand{\phi}{\varphi}
\renewcommand{\le}{\leqslant}
\renewcommand{\ge}{\geqslant}
\renewcommand{\leq}{\leqslant}
\renewcommand{\geq}{\geqslant}

\newcommand{\ainc}[1]{\hyperref[ainc]{{\normalfont(aInc){\ensuremath{_{#1}}}}}}
\newcommand{\adec}[1]{\hyperref[adec]{{\normalfont(aDec){\ensuremath{_{#1}}}}}}
\newcommand{\inc}[1]{\hyperref[ainc]{{\normalfont(Inc){\ensuremath{_{#1}}}}}}
\newcommand{\dec}[1]{\hyperref[adec]{{\normalfont(Dec){\ensuremath{_{#1}}}}}}
\newcommand{\azero}{\hyperref[azero]{{\normalfont(A0)}}}
\newcommand{\aone}{\hyperref[aone]{{\normalfont(A1)}}}

\newcommand{\aonen}[1]{\hyperref[aonen]{{\normalfont(A1-\ensuremath{#1})}}}
\newcommand{\VAn}[1]{\hyperref[VA1n]{{\normalfont(VA1-\ensuremath{#1})}}}
\newcommand{\wVAn}[1]{\hyperref[wVA1n]{{\normalfont(wVA1-\ensuremath{#1})}}}
\newcommand{\sVAn}[1]{\hyperref[sVA1n]{{\normalfont(sVA1-\ensuremath{#1})}}}

\newcommand{\Phiw}{\Phi_{\rm{w}}}

\begin{document}

\title{Higher integrability for parabolic PDEs with generalized Orlicz growth}

%
\author{Peter H\"ast\"o}
\address{Department of Mathematics and Statistics, FI-00014 University of Helsinki, Finland
\newline
and
\newline
Department of Mathematics and Statistics, FI-20014 University of Turku, Finland}
\email{peter.hasto@helsinki.fi}

\author{Jihoon Ok}
\address{Department of Mathematics, Institute for Mathematical and Data Science, Sogang University, Seoul 04107, Republic of Korea}
\email{jihoonok@sogang.ac.kr}
%




\begin{abstract}
We prove higher integrability of the gradient of weak solutions 
to nonlinear parabolic systems whose prototype is 
\[
\partial_t u-\mathrm{div}\Big(\frac{\varphi'(z, |\nabla u|)}{|\nabla u|}\nabla u\Big) =0, \qquad u=(u^1,\dots,u^N),
\]
where $\varphi$ is a generalized Young function. 
Special cases of our main theorem include previously known results 
for the $p$-growth, the variable exponent and the double phase growth. 
Also included are previously unknown borderline double phase growth and 
perturbed variable exponent growth, among others. 
The problem is controlled by a natural requirement of comparison of $\varphi$
between points in intrinsic parabolic cylinders via an (A1)-condition, 
which unifies disparate conditions from the special cases. {Moreover, we handle both the singular and degenerate cases at the same time, providing a unified proof of a reverse H\"older type inequality.}

\end{abstract}

\maketitle


\section{Introduction}\label{sec1}

In this paper, we prove higher integrability of the gradient of weak solutions to nonlinear parabolic systems of the form
\begin{equation}\label{maineq}
\partial_t u-\mathrm{div}(A(z,u,\nabla u))=0\quad \text{in }\ \Omega_T=\Omega\times (0,T] \subset \R^n\times \R,
\end{equation}
where $\Omega\subset\R^n$ with $n\geq 2$ is open, $z=(x,t)\in \Omega_T$, $u=(u^1,\dots,u^N) : \Omega_T \to \R^N$, and $\nabla u:=D_xu\in \R^{N \times n}$ is the (spatial) gradient of $u$. We assume that the nonlinearities
 $A=(A_{ij}):\Omega_T\times \R^N\times \R^{N\times n} \to \R^{N\times n}$, 
 where $1\le i\le N$ and $1\le j \le n$, satisfy
 the generalized Orlicz growth conditions
\begin{equation}\label{A-condition}
|A(z,v,\xi)| \leq \Lambda \frac{\phi(z,|\xi|)}{|\xi|}
\quad \text{and}\quad
\underbrace{\sum_{i=1}^N\sum_{j=1}^n A_{ij}(z,v,\xi) \xi_{ij}}_{=:\langle A(z,v,\xi),\xi \rangle} 
\ge 
\frac1\Lambda \phi(z,|\xi|)
\end{equation}
for all $z\in\Omega_T$, $v\in \R^N$ and $\xi=(\xi_{ij}) \in \R^{N\times n}$ and for some $\Lambda\ge 1$, where $\phi:\Omega_T\times [0,\infty)\to[0,\infty)$ is a weak $\Phi$-function that satisfies 
\azero{},
\ainc{p} and \adec{q} with 
\[
\frac{2n}{n+2} < p \leq q.
\]
The definitions of weak $\Phi$-function, 
\azero{},
\ainc{} and \adec{} will be introduced in the next section. Note that the lower bound $\frac{2n}{n+2}$ is generally assumed in parabolic regularity theory; see \cite{DiBenedetto_book} and also \cite{KiLewis00}. 
The prototype of \eqref{maineq} satisfying the generalized Orlicz growth condition is the following $\phi$-parabolic system:
\begin{equation}\label{eq:prototype}
\partial_t u-\mathrm{div}\left(\frac{\phi'(z, |\nabla u|)}{|\nabla u|}\nabla u\right) =0.
\end{equation}
We note that the growth condition \eqref{A-condition} is equivalent to the so-called quasi-isotropic 
$(p,q)$-growth and coercivity conditions; see \cite{HasO22b} (in particular, Remark 4.3), where we 
proved higher integrability results for elliptic equations corresponding to \eqref{maineq} with growth \eqref{A-condition}.

When $\phi$ is independent of $z$, i.e. $\phi(z,s)\equiv\phi(s)$, the regularity theory for equations and systems of type \eqref{eq:prototype} has been extensively studied. In the classical $p$-growth setting, we refer to the monographs of DiBenedetto, Gianazza and Vespri \cite{DiBenedetto_book,DiBeGiaVe_book} and the references therein. For the more general Orlicz growth case, Lieberman \cite{Lie06} proved Hölder continuity of the gradient for Lipschitz continuous solutions to the parabolic system \eqref{eq:prototype}. Subsequently, Baroni and Lindfors \cite{BarL17} as well as Diening, Scharle and Schwarzacher \cite{DieScharSchwa19} independently established the local boundedness of the gradient under restrictive settings and assumptions, but within a unified framework treating both the singular and degenerate cases. 
Recently, the second author, Scilla and Stroffolini \cite{OSS24} obtained local boundedness results for the system in full generality.
As for Hölder continuity, Hwang and Lieberman \cite{HwangLie15-1, HwangLie15-2} have studied bounded solutions of parabolic equations in the Orlicz setting, proving Hölder continuity in the regimes $p\le q\le2$ and $2\le p\le q$. However, their results do not apply to the general case $p<2<q$, 
where a unified approach remains an open problem.

The higher integrability of nonlinear parabolic systems satisfying growth conditions involving special types of growth functions $\phi(z,s)$ has been
studied over the last two decades. Kinnunen and Lewis \cite{KiLewis00} considered the $p$-growth case, i.e., $\phi(z,s) \equiv s^p$ with $\frac{2n}{n+2} < p < \infty$, and proved that $\phi(\cdot, |\nabla u|) = |\nabla u|^p \in L^{1+\epsilon}_{\loc}(\Omega_T)$ for some $\epsilon > 0$. Following this pioneering work, we extended the results to the Orlicz growth case $\phi(z,s) \equiv \phi(s)$
in \cite{HasO21}, where the growth function $\phi$ is independent of the variable $z$.
It is worth noting that many parts of the proofs in \cite{KiLewis00} rely on different analyses for the degenerate case ($p \geq 2$) and the singular case ($\frac{2n}{n+2} < p < 2$). In contrast, the approach in \cite{HasO21} is 
uniform and independent of the ranges of the growth exponents $p$ and $q$. We further refer to \cite{BoDuMa13} for higher integrability results for autonomous parabolic systems with $(p,q)$-growth.

Moreover, in the non-autonomous setting where $\phi$ depends on $z$,
Bögelein and Duzaar \cite{BoDu11} established higher integrability for the variable exponent case $\phi(z,s) = s^{p(z)}$ with $\frac{2n}{n+2} < p^- \leq p(z) \leq p^+$.
See also the earlier scalar result of Antontsev and Zhikov \cite{AnZhi05}.

Recently, Kim, Kinnunen and Moring \cite{KimKM23} investigated the double phase growth function 
$\phi(z,s) = s^p + a(z)s^q$ with $0 \leq a(z) \leq L$ from \cite{ColM15a, Zhi86} in 
the parabolic case, and proved higher integrability in the degenerate case $p\ge 2$. This result was 
extended to the singular case $\frac{2n}{n+2} < p < 2$ by Kim and Särkiö \cite{KimS24}. 
The papers \cite{KimKM23, KimS24} are based on the unnatural a priori assumption 
$|Du|^q\in L^{1}(\Omega_T)$, which was reduced to the natural energy space 
using a Lipschitz-truncation technique by Kim, Kinnunen and Särkiö \cite{KimKS25}. 
We stress that our result covers all of these special cases as well 
as perturbed variants that have not previously been obtained, see Section~\ref{sect:aone}, 
and that we avoid the bootstrapping via the Lipschitz truncation.
Related higher integrability results for parabolic double phase type problems were obtained in \cite{KimO24,OhSen26, Sen25}.

We now state our higher integrability result for parabolic systems with generalized Orlicz growth. Let $\phi:\Omega_T
\times[0,\infty)\to[0,\infty)$ be a generalized Orlicz function, with notation and structural conditions to be introduced in Sections~\ref{sect:preliminaries} and~\ref{sect:aone}. For $z\in \Omega_T$ and $s\ge 0$, we denote
\[
\mathcal{D}(z,s):=\min\left\{\phi^{-1}(z,s)^2,s^{\frac{n}{2}+1}\phi^{-1}(z,s)^{-n}\right\}.
\]

\begin{remark}
If $\phi$ satisfies \ainc{2}, corresponding to the purely degenerate case, 
then $\mathcal{D}(z,s) \approx \phi^{-1}(z,s)^2$ for all $z\in \Omega_T$ and $s\ge1$.
\end{remark}

We say that $u\in L^\infty_{\loc}((0,T];L^2_\loc(\Omega,\R^N))\cap L^1_\loc ((0,T];W^{1,1}_\loc (\Omega,\R^N))$ with $\phi(\cdot,|\nabla u|)\in L^1_\loc (\Omega_T)$ is a local \textit{weak solution} to \eqref{maineq}
if
\begin{equation}\label{weakform}
-\int_{\Omega_T} u \cdot \partial_t \zeta \, dz +\int_{\Omega_T} \langle A(z,u,\nabla u),\nabla\zeta\rangle \,dz =0
\end{equation}
for all $\zeta=(\zeta^1,\dots,\zeta^N)\in C^\infty_0(\Omega_T,\R^N)$. Note that, by approximation, we can also consider test functions $\zeta$ that are Lipschitz and compactly supported in $\Omega_T$.
The natural energy for this problem is $\phi(\cdot, |\nabla u|) \in L^1_\loc(\Omega_T)$, 
and better integrability of $\phi(\cdot, |\nabla u|)$ is referred to as higher integrability. 
Our main result is the following. 

\begin{theorem}[Higher integrability]\label{thm:main}
Let $\phi\in \Phiw(\Omega_T)$ satisfy \azero{}, intrinsic parabolic \aone{}, \ainc{p} and \adec{q} with $\frac{2n}{n+2} < p \leq q$ and constant $L\geq 1$, and $u$ be a local
weak solution to the system \eqref{maineq} satisfying \eqref{A-condition}. Suppose that 
{$q<p+2 \min\{\frac{p}{n},1\}$.}
There exists $\epsilon=\epsilon(n,N,p,q,\Lambda,L)>0$ such that $\phi(\cdot,|\nabla u|)\in L^{1+\epsilon}_{\loc}(\Omega_T)$ with the following estimate: for any parabolic cylinder $Q_{2r}=Q_{2r}(z_0)\Subset\Omega_T$ 
with $\int_{Q_{2r}} \phi(z,|\nabla u|) +1 \,dz \le 1$,
\[
\fint_{Q_{r}} \phi(z,|\nabla u|)^{1+\epsilon}\,dz 
\le 
c \left[ (\mathcal{D}^-_{Q_{2r}})^{-1}\bigg(\fint_{Q_{2r}}[\phi(z, |\nabla u|)+ 1]\,dz\bigg)\right]^\epsilon\fint_{Q_{2r}}[\phi(z, |\nabla u|)+ 1]\,dz
\]
for some $c=c(n,N,p,q,\Lambda,L)>0$, where 
$\mathcal{D}_{Q_{2r}}^-(s):=\inf_{z\in Q_{2r}}\mathcal{D}(z,s)$ and $(\mathcal{D}_{Q_{2r}}^-)^{-1}(s)$ is the left-inverse of $\mathcal{D}_{Q_{2r}}^-(s)$.
\end{theorem}

\begin{remark}
In Theorem~\ref{thm:main}, if $\phi$ is independent of $z$, then the restriction 
$q<p+\min\{\frac{2p}{n},2\}$
is not necessary. In fact, this condition is only used in Lemma~\ref{lem:poincare}, see Remark~\ref{rmk:poincare}. The case 
$q \ge p+2 \min\{\frac{p}{n},1\}$ remains open. 
\end{remark}

{\begin{remark}We also note that $q \le p + \alpha (\min\{\frac{p}{2}, 1\} - \frac{n}{n+2})$ for some $\alpha \in (0,2]$ implies the additional condition $q < p + 2\min\{\frac{p}{n}, 1\}$. 
When $\alpha\le 1$, this condition was used in the double phase case $\phi(z,s)=s^p+a(z)s^q$ with an $\alpha$-H\"older continuous function $a$ with respect to the parabolic distance in \cite{KimKM23, KimS24}. 
Consequently, our results include their results as special cases and in fact extend them 
to the case where $a$ is $C^{1, \alpha-1}$ in $x$ and $C^{0,\frac{\alpha}{2}}$ in $t$ for $1<\alpha \le 2$; for further details, see Proposition~\ref{prop:DP} 
and Remark~\ref{rem:borowski}.
\end{remark}}

In Section~\ref{sect:aone}, we show that the higher integrability result in Theorem~\ref{thm:main} recovers known results for nonstandard growth settings, such as the $p(z)$-growth case and the 
double phase case. Moreover, our result also applies to their perturbations, for example, $\phi(z,s)=s^{p(z)}\log(e+s)$ or $s^p+a(z)s^p\log(e+s)$, cf.\ Proposition~\ref{prop:borderline}.
Regarding the sharpness of the higher integrability estimate in the theorem, 
suppose that $\phi(z,s)=s^{p(z)}$. If $p$ is log-H\"older continuous, then for $1\le s \le |Q_{2r}|^{-1}$ we can see that
\[
s^{p^-_{Q_{2r}}}\approx s^{p^+_{Q_{2r}}}
\quad\text{and hence}\quad (\mathcal{D}_{Q_{2r}}^-)^{-1}(s) \approx \max\{s^{\frac{p(z_0)}{2}},\,s^{\frac{2p(z_0)}{(n+2) p(z_0)-2n}}\},
\]
see the proof of Proposition~\ref{prop:varExpo}.
In this case, the estimate recovers exactly the higher integrability estimate for the $p(z)$-growth setting in \cite{BoDu11}. Furthermore, our estimate is sharper than the 
best previously known estimates for the double phase case obtained in \cite{KimKM23,KimS24}.

The proof of Theorem~\ref{thm:main} is based on the method of Kinnunen and Lewis \cite{KiLewis00}. We first establish Caccioppoli-type estimates (Theorem~\ref{lem:caccio}), which, together with the Sobolev--Poincaré inequality, yield reverse Hölder-type inequalities for the gradient of weak solutions to \eqref{maineq} in intrinsic parabolic cylinders under a suitable balancing condition (Lemma~\ref{lem:reverse}). We then estimate super-level sets of the function $\phi(z,|Du|)+1$ by means of a stopping-time argument, together with a Vitali-type covering lemma (Lemma~\ref{lem:vitali}) and the reverse Hölder inequality. Finally, applying Fubini’s theorem, we obtain the higher integrability estimate.

 We would like to emphasize the novelty of our approach:
\begin{itemize}
\item 
As mentioned above, we provide a unified approach that covers not only the degenerate and singular cases but also intermediate situations that cannot be clearly classified into either category. In particular, in contrast with the double phase setting considered in \cite{KimKM23,KimS24}, our proof does not require distinguishing between cylinders according to their growth behavior. 
Therefore, we are able to present a simpler proof even when dealing with a more general class of growth conditions. 
We expect this to enable future extensions such as the study of bounded solutions following 
\cite{BarCM18, ColM15b} in the elliptic case.
\item 
The Steklov average method, which is a standard tool to derive Caccioppoli-type estimates in parabolic problems, turns out to be unsuitable in the case of nonstandard growth. 
For this reason, the first paper \cite{KimKM23} on parabolic double phase problems 
($\phi(z,s)=s^p+a(z)s^q$) assumed stronger a priori integrability of the solution, which was then relaxed via Lipschitz truncation in \cite{KimKS25}, as mentioned above.
The fundamental obstruction to the Steklov average in the nonstandard growth setting 
is that $\iint_{Q_{r,\rho}(x_0,t_0)} \phi(x,t,|Du|)\, dx\, dt \le 1$
does not imply an estimate 
$\int_{t_0-\rho}^{t_0+\rho} \phi(x,t,|Du|)\, dt \le M$ 
on thick time-slices.
This prevents the use of an \aone{}-type property with respect to $t$. 
In this paper, we propose an alternative approach: the mollification with respect to both $x$ and $t$ variables simultaneously. The convergence of this approximation procedure is guaranteed by the parabolic version of the \aone{}-condition, see Section~\ref{sect:caccioppoli}.
\item
{We improve the higher integrability estimate in Theorem~\ref{thm:main} 
compared to the results for double phase problems, where $\phi(z,s)=s^p+ a(z)s^q$, in \cite{KimKM23, KimS24}. 
In those previous works, the constants $\epsilon$ and $c$ depend on 
$\|u\|_{L^\infty(0,T;L^2(\Omega))}$, whereas our estimate is independent of this norm. This improvement stems from an enhanced technique for obtaining the reverse H\"{o}lder inequality, specifically within the proof of Lemma~\ref{lem:poincare}. 
Furthermore, we extend the range of coefficients $a$ to $C^{1,\alpha}$ in $x$; 
see Remark~\ref{rem:borowski}.}
\end{itemize}

The paper is organized as follows. In the next section, Section \ref{sect:preliminaries}, 
we introduce notation and results concerning Orlicz functions. 
Then in Section~\ref{sect:aone} we introduce our main condition regarding regularity 
in the $z$-direction, \aone{}, and show how it covers and unifies 
earlier results as special cases of this general framework. 
In Section~\ref{sect:caccioppoli}, we obtain a Caccioppoli inequality for our problem. 
In Section \ref{sect:inequalities}, we derive Sobolev--Poincar\'e and reverse H\"older inequalities. 
Finally in the last section, Section \ref{sect:proof}, we prove the main result, Theorem~\ref{thm:main}.


\section{Preliminaries}\label{sect:preliminaries}

For $z_0=(x_0,t_0)\in \mr^n\times \mr$ and $r,\rho>0$, we define the cylinder in $\R^{n+1}$ by
\[
Q_{r,\rho}(z_0):=B_r(x_0)\times (t_0-\rho,t_0+\rho),
\]
where $B_r(x_0)$ is the open ball in $\R^n$ with center $x_0$ and radius $r$. The 
\emph{parabolic cylinder} is 
\[
Q_r(z_0):=Q_{r,r^2}(z_0)=B_r(x_0)\times (t_0-r^2,t_0+r^2),
\]
and the \emph{intrinsic parabolic cylinder} (associated with the generalized Orlicz function $\phi$) is
\[
\begin{split}
Q^{\lambda}_{r}(z_0) & :=Q_{r,\frac{\phi^{-1}(z_0,\lambda)^2}{\lambda}r^2}(z_0)\\
&\ =B_r(x_0)\times I^\lambda_r(z_0),
\qquad
\text{where }\ I^\lambda_r(z_0):= (t_0-\tfrac{\phi^{-1}(z_0,\lambda)^2}{\lambda}r^2,
t_0+\tfrac{\phi^{-1}(z_0,\lambda)^2}{\lambda}r^2).
\end{split}
\]
If the center of a cylinder is not important, we may omit it from the notation.

Let $f,g:[0,\infty)\to [0,\infty)$. The function $f$ is said to be \textit{almost increasing} if there exists $L\ge 1$ such that $f(s_1)\leq Lf(s_2)$ for all $0<s_1<s_2<\infty$. If $L=1$ we say $f$ is \emph{increasing}. 
\emph{Almost decreasing} and \emph{decreasing} are defined analogously. 
We write $g\lesssim f$ if there exists $L\geq 1$, depending only on structural constants, such that $g(s)\leq L f(s)$ for all $s\geq 0$. We say that $f$ and $g$ are equivalent, and write $f\approx g$, if both $g\lesssim f$ and $f\lesssim g$.

We denote the average of $v\in L^1(U,\R^N)$ with $0<|U|<\infty$ by $v_{U} := \fint_U v\, dz := \frac{1}{|U|}\int_U v \,dz$; 
if $U=Q^\lambda_{r}$, we abbreviate $(v)^\lambda_r:= v_{Q^\lambda_{r}}$.

The parabolic cylinder is a ball in the metric 
$d_{\mathrm p}(z, z')=\max\{|x-x'|, \sqrt{|t-t'|}\}$. By \textit{(parabolic) Lebesgue points} 
of $f \in L^1_{\mathrm{loc}}(\mathbb{R}^{n+1})$
we mean Lebesgue points in the metric space $(\R^{n+1}, d_{\mathrm p}, \mathcal{L}^{n+1})$, i.e.\ 
points satisfying 
\[
\lim_{r\to 0} \fint_{Q_r(z_0)} f(z)\, dz = f(z_0).
\]
Since the $(n+1)$-dimensional Lebesgue measure $\mathcal{L}^{n+1}$ is doubling with respect 
to the parabolic metric (specifically, $|Q_{2r}| = 2^{n+2} |Q_r|$), 
by the Lebesgue differentiation theorem (e.g. \cite[Theorem 1.8]{Heinonen_book}), $\mathcal{L}^{n+1}$-almost every point $z_0$ is a Lebesgue point of $f$
and the Lebesgue points satisfy 
\[
\lim_{r\to 0^+} \fint_{Q_r(z_0)} |f(z)-f(z_0)|\, dz = 0.
\]
Assume that $z_0$ is a Lebesgue point. 
Consider general cylinders, including the intrinsic parabolic cylinders, 
and note that
$Q_{r, a r^2}(z_0)\subset Q_{\max\{1,\sqrt a\}r}(z_0)$. Then
\[
\fint_{Q_{r, a r^2}(z_0)} |f(z)-f(z_0)|\, dz
\le
\frac{|Q_{\max\{1, \sqrt a\}r}(z_0)|}{|Q_{r, a r^2}(z_0)|}
\fint_{Q_{\max\{1, \sqrt a\}r}(z_0)} |f(z)-f(z_0)|\, dz.
\] 
The first factor on the right-hand side simplifies to $\frac{(\max\{1, \sqrt a\}r)^{n+2}}{ar^{n+2}}
= \max\{a^{-1}, a^{n/2}\}$ which is independent of $r$. Thus it follows that 
\[
\lim_{r\to 0^+} \fint_{Q_{r, a r^2}(z_0)} |f(z)-f(z_0)|\, dz = 0,
\]
so $z_0$ is also a Lebesgue point with respect to intrinsic parabolic cylinders and related ones.

We introduce fundamental conditions on generalized Orlicz functions.
These conditions allow us to work easily with weak $\Phi$-functions, without 
worrying about convexity.

\begin{definition} 
Let $\phi:U\times [0,\infty)\to[0,\infty)$, where $U\subset \R^m$, and $p,q>0$. 
\begin{itemize}
\item[\normalfont(A0)]\label{azero} 
There exists a constant $L\geq 1$ such that $L^{-1}\le \phi(z,1)\le L$ for every $z\in U$.
\item[\normalfont(aInc)$_p$]\label{ainc} 
The map $s\mapsto \phi(z,s)/s^p$ is almost increasing on $(0,\infty)$ with a constant $L\geq 1$ independent of $z\in U$.
\vspace{0.2cm}
\item[\normalfont(aDec)$_q$]\label{adec} 
The map $s\mapsto \phi(z,s)/s^q$ is almost decreasing on $(0,\infty)$ with a constant $L\geq 1$ independent of $z\in U$.
\end{itemize}
In the case $L=1$, we use \inc{} and \dec{} instead of \ainc{} and \adec{}, respectively.

We also define \ainc{} and \adec{} for functions $\psi:[0,\infty)\to [0,\infty)$, that is, $\psi=\psi(s)$ is independent of $z$, in the same way.
\end{definition}

Note that \ainc{p} or \adec{q} for some $1<p\leq q$ is equivalent to the $\nabla_2$- or $\Delta_2$-conditions, respectively, see \cite[Proposition~2.2.6]{HarH19}. 
From the above definition, we see that \ainc{p} implies \ainc{\tilde p} 
for all $\tilde p<p$ and \adec{q} implies \adec{\tilde q} for all $\tilde q>q$. 
Moreover, if $\phi$ satisfies \ainc{p} and \adec{q}, then for any $z\in U$, $s\in(0,\infty)$ and $0<c <1 <C$,
\[
c^q L^{-1} \phi(z,s) \leq \phi(z,cs)\leq c^pL \phi(z,s)
\quad\text{and}\quad 
C^pL^{-1} \phi(z,s) \leq \phi(z,Cs)\leq C^qL \phi(z,s).
\]
We shall use these inequalities numerous times later without explicit mention. We further define 
\[
\phi^+_U(s):= \sup_{z\in U}\phi(z,s)
\quad \text{and}\quad
\phi^-_U(s):= \inf_{z\in U}\phi(z,s).
\]
If $\phi$ satisfies \azero, \ainc{p} and \adec{q}, then $\phi^\pm_U$ are finite and satisfy \azero{}, 
\ainc{p} and \adec{q} with the same constants.

\begin{definition} 
Let $\phi:U\times [0,\infty)\to[0,\infty)$ where $U\subset\R^m$. 
We say that $\phi$ is a \textit{$\Phi$-prefunction} if for every measurable function $f:U\to \R$ the map $z\mapsto \phi(z,|f(z)|)$ is measurable and for almost every $z\in U$ the map $s\mapsto \phi(z,s)$ is increasing with $\phi(z,0)=0$, $\lim_{s\to 0^+}\phi(z,s)=0$, $\lim_{s\to\infty}\phi(z,s)=\infty$.
A $\Phi$-prefunction $\phi$ which satisfies \ainc{1} is called a \emph{weak $\Phi$-function}, 
denoted $\phi\in\Phiw(U)$.
\end{definition}

As an example of the robustness of this definition, we note that $\sqrt{\phi(z,s^2)}$ need not be 
convex if $\phi$ is, but the \ainc{1} property is conserved.
Moreover, the condition \ainc{1} captures some essential features of convexity, as it allows us 
to use the following Jensen-type inequality. 

\begin{lemma}[Jensen inequality, Lemma~4.3.2, \cite{HarH19}] \label{lem:Jensen}
If $\psi:[0,\infty)\to [0,\infty)$ is increasing with $\psi(0)=0$ and satisfies 
\ainc{1} with constant $L\geq 1$, then 
\[
\psi\bigg( \frac{1}{L^2} \fint_U |f|\, dz\bigg) \le \fint_U \psi(|f|)\, dz.
\] 
\end{lemma}

Also, we can use the conditions effectively with Young-type inequalities 
and inverse functions. We recall the 
definition of the \emph{conjugate weak $\Phi$-function}:
\[
\phi^*(z,s) := \sup_{\tilde s \ge 0} \,(s \tilde s - \phi(z,\tilde s)).
\]
This definition directly implies Young's inequality:
\[
s \tilde s \leq \phi(z,s)+\phi^*(z,\tilde s),\qquad s,\tilde s \geq 0.
\]
The exact value of $\phi^*$ can usually not be easily determined, but we have the following 
useful estimate which can be found in the proof of \cite[Theorem~2.4.8]{HarH19}:
\[
\phi^*\big(z,\tfrac{\phi(z,s)}{Ls}\big) \leq \phi(z,s) \leq \phi^*\big(z,2\tfrac{\phi(z,s)}{s}\big),
\]
where $L$ is the constant from \ainc{1}. This will be used multiple times in what follows.

\begin{remark}\label{rmkphiphic}
Suppose that $\phi$ is a weak $\Phi$-function which satisfies 
\ainc{p} and \adec{q} with $1\leq p\leq q$. We can define 
\[
\tphi(z,s)=\int_0^s \tilde s^{p-1}\sup_{\sigma\in (0,\tilde s]}\frac{\phi(z,\sigma)}{\sigma^p}\,d\tilde s. 
\]
Then $\tphi\approx \phi$, and $\tphi(z,s)$ is strictly increasing, differentiable almost everywhere and convex in $s$ (since its almost-everywhere derivative
is increasing) and also satisfies \inc{p} and \adec{q}. Moreover, by \cite[Lemma 2.2.6]{HarH19}, $\tphi$ satisfies \dec{\tilde q} for some $\tilde q\ge q$ depending on $p$, $q$ and $L$. 
Since the claims that 
we are proving are invariant under equivalence of weak $\Phi$-functions, we may thus assume when necessary that 
$\phi$ is strictly increasing and differentiable almost everywhere, and satisfies \inc{p} and \dec{\tilde q} for some $\tilde q\ge1$ depending on $p$, $q$ and $L$. 
\end{remark}

We next introduce the \emph{left-inverse} of $\phi\in \Phiw(U)$:
$$
\phi^{-1}(z,\tilde s):=\inf\{s \ge 0: \phi(z,s)\ge \tilde s\}.
$$ 
Clearly, $\phi^{-1}(z,\phi(z,s))\le s$ and, if $\phi$ is continuous, $\phi(z, \phi^{-1}(z,\tilde s))\geq \tilde s$. 
We first observe that 
\[
L^{-2}\le \phi^{-1}(z,1) \le L^2, 
\quad \text{if $\phi$ satisfies \azero{},}
\]
where $L\ge 1$ is the constant in both \ainc{1} and \azero{}, and that by \cite[Theorem 2.4.8]{HarH19}, $\phi^{-1}(z,s)(\phi^*)^{-1}(z,s)\approx s$.
We further note that in view of Remark~\ref{rmkphiphic}, $(\phi\circ \phi^{-1})(z,s)\approx (\phi^{-1}\circ \phi)(z,s)\approx s$ if $\phi$ satisfies \adec{q} with $q\ge 1$.
By \cite[Proposition~2.3.7]{HarH19}, $\phi^{-1}$ satisfies \ainc{1/q} or \adec{1/p} if and 
only if $\phi$ satisfies \adec{q} or \ainc{p}, respectively. 
From these facts and Lemma~\ref{lem:Jensen} we conclude the Jensen-inequality 
\[
\fint_U \psi(|f|)\, dz \le c \psi\bigg(\fint_U |f|\, dz \bigg)
\]
when $\psi$ satisfies \adec{1}.


\section{Intrinsic parabolic (A1) and examples}\label{sect:aone}

We introduce the main assumption on weak $\Phi$-functions that we use 
to obtain the higher integrability result. 
It is analogous to the elliptic \aone{} assumption, except that we use an 
intrinsic parabolic cylinder as our test set. The difference seems minor, but the 
condition $\lambda\le |Q_r^\lambda|^{-1}$ is quite subtle, since also the right-hand side 
depends on $\lambda$. However, this is needed to capture both the singular and degenerate cases 
in one, as we will see in the examples. 

\begin{definition}\label{aone}
We say that $\phi\in \Phiw(\Omega_T)$ satisfies \textit{intrinsic parabolic} 
\aone{} 
if there exists 
$\beta\in (0,1]$ 
such that for every $Q:=Q_r^\lambda(z_0) \Subset \Omega_T$ 
with
$\lambda \in [1, |Q|^{-1}]$
we have
\[
\phi(z,\beta s)\le \phi(w, s) 
\quad \text{whenever } \ z,w\in Q\ \text{ and } \ \phi(w,s)\in [1,|Q|^{-1}].
\]
\end{definition}

Unless otherwise specified, in this paper, the \aone{} condition will always refer to the intrinsic parabolic \aone{} condition.

\begin{remark}\label{rmk:A1} 
Suppose $\phi\in \Phiw(\Omega_T)$ satisfies \azero{}. Then the inequality in the definition of \aone{} can be equivalently written as 
\[
\tilde\beta \phi^{-1}( z,\tilde s) \le \phi^{-1}(w,\tilde s) 
\quad \text{whenever }\ z,w\in Q \ \text{ and }\ \tilde s\in [1,|Q|^{-1}],
\] 
for some $\tilde \beta \in (0,1]$.
This can be proved in the same manner as \cite[Corollary 4.1.6]{HarH19}.
\end{remark}

\begin{remark}\label{rmk:A11}
Suppose $\phi\in \Phiw(\Omega_T)$ satisfies \azero{} and \adec{q} with $q\ge1$. Then the inequality from 
\aone{} can be equivalently written as 
\[
\phi_Q^+ (s)\lesssim \phi_Q^- (s)
\quad \text{whenever } \ \phi_Q^-(s)\in [1,|Q|^{-1}].
\]
To see this, we first suppose that the \aone{}-inequality holds and $\phi^-_Q(s)\in [1,|Q|^{-1}]$. Then there exists $z,w\in Q$ such that $2\phi(z,s)\ge \phi^+_Q(s)$ and $\phi(w,s)\le 2\phi^-_{Q}(s)$. Thus, by \ainc{1} of $\phi$, $\phi(w,(2L)^{-1}s)\le 2^{-1} \phi(w,s) \le \phi^-_{Q}(s)\le |Q|^{-1}$. 
Therefore, by using \azero{}, \aone{}, \ainc{1} and \adec{q} of 
$\phi$, and the fact that $1\le \phi(w,s)$, we obtain
\[
\phi^+_Q(s)\le 2\phi(z,s) \lesssim \phi(z,\beta (2L)^{-1}s) \lesssim \phi(w, (2L)^{-1}s)+1 \lesssim \phi(w,s) \lesssim \phi^-_{Q}(s).
\]
The opposite implication follows in a similar way.
\end{remark}

\begin{remark}\label{rmk:A1para}
Suppose $\phi\in \Phiw(\Omega_T)$ satisfies \azero{} and the intrinsic parabolic \aone{}. 
We show the local form of the \emph{parabolic} \aone{} condition wherein the inequality holds on $Q_r(z_0)$ whenever $Q_{2r}(z_0)\Subset\Omega_T$.
Fix a parabolic cylinder $Q_r(z_0) \Subset \Omega_T$ with $Q_{2r}(z_0)\Subset\Omega_T$ and $1\le |\mathcal C|^{-1}$, and 
set $\rho:=L^{-2}r$. Then $L^{-8}r^2 \le \phi^{-1}(z,1)^2\rho^{2}\le r^2$ by \azero{}. 
There exists $k_0\in \mathbb N$ depending only on $n$ and $L$ such that any $z,w\in Q_r(z_0) $ 
can be connected by a chain of $k_0$ intrinsic parabolic cylinders $Q^1_\rho(z_k)\subset Q_{2r}(z_0)$,
i.e.\ $z\in Q^1_\rho(z_1)$, $w\in Q^1_\rho(z_{k_0})$, and $ Q^1_\rho(z_{k})\cap Q^1_\rho(z_{k+1})\neq \emptyset$ for all $k=1,\dots,k_0-1$. 
Since $| Q^1_\rho(z_k)|\le 2\,|B_1|\rho^{n+2}\phi^{-1}(z_k,1)^2 \le 2\,|B_1|r^{n+2}=|Q_r(z_0) |$ for every $k$, Remark~\ref{rmk:A1} and a chain argument imply that $\beta^{k_0}\phi^{-1}(z, \tilde s) \le \phi^{-1}( w, \tilde s)$ for all $\tilde s\in [1,|Q_r(z_0) |^{-1}]$.
\end{remark}

\begin{remark}\label{rmk:phi*A1}
Suppose that $\phi\in \Phiw(\Omega_T)$ satisfies \azero{}, \ainc{p} and \adec{q} with $1<p\le q$. 
If $\phi$ satisfies \aone{}, then $\phi^*$ satisfies an \aone{}-type condition 
with intrinsic cylinders associated with $\phi$, not $\phi^*$. This follows by the same proof as 
\cite[Lemma~4.1.7]{HarH19} in view of Remark~\ref{rmk:A1}.
\end{remark}

Next, we present some examples of generalized Orlicz functions and their conditions that 
imply the intrinsic parabolic \aone{}. Note that all the following examples satisfy \azero{} and \adec{q}, 
hence we will utilize the formulation from Remark~\ref{rmk:A11} instead of working directly with the definition. 
The following result exactly captures the distinct conditions of \cite{KimKM23, KimS24} 
and shows how \aone{} captures both the singular and degenerate case via the requirement 
$\lambda\le |Q|^{-1}$. 

\begin{proposition}[Double phase]\label{prop:DP}
Consider $\phi(z, s):= s^p + a(z) s^q$, where
$a(z)=a(x,t)$ is a nonnegative and bounded function, $\alpha$-H\"older continuous in $x$ and $\frac\alpha2$-H\"older continuous in $t$ for some $\alpha\in(0,1]$. 
Assume that 
\begin{equation} \label{eq:double_range}
\tfrac{2n}{n+2}<p\le q\le p + \alpha\, \big(\!\min\{\tfrac p2,1\}-\tfrac n{n+2}\big).
\end{equation}
Then $\phi$ satisfies the intrinsic parabolic \aone{} condition. 
\end{proposition}
\begin{proof}
Let $\lambda\in [1, |Q|^{-1}]$ for $Q=Q_{r,br^2}(z_0)$ with 
$b := \frac{\phi^{-1}(z_0,\lambda)^2}\lambda$.
It is enough to show the following sufficient condition for \aone{}:
\[
s^p + a_Q^+s^q = \phi_Q^+ (s)\lesssim \phi_Q^- (s) =s^p + a_Q^-s^q
\quad \text{whenever } \ \phi_Q^-(s)\le |Q|^{-1}.
\]
The $s^p$ on the left-hand side is irrelevant, so the inequality 
can be reduced to $a_Q^+-a_Q^-\lesssim s^{p-q}$. 
We will prove below that 
\begin{equation}\label{doublephase-sa1}
a_Q^+-a_Q^-\lesssim s^{p-q}
\quad \text{whenever } \ \phi(z_0, s)\le |Q|^{-1}.
\end{equation}
To connect this to the sufficient condition, it only remains to consider $s>0$ with 
$\phi^-_Q(s)\le |Q|^{-1}< \phi(z_0,s)$. Let $\gamma\ge 1$ be such that 
$\phi(z_0,s)=\gamma^p |Q|^{-1}$, which implies $\phi(z_0,\frac{s}{\gamma})\le |Q|^{-1}$. By \eqref{doublephase-sa1}, 
\[
\frac{\gamma^p}{|Q|} \le \phi^+_Q(s) 
\le \gamma^{q}\phi^+_Q\Big(\frac{s}{\gamma}\Big) 
\lesssim \gamma^{q} \phi^-_Q\Big(\frac{s}{\gamma}\Big) 
\lesssim \gamma^{q-p} \phi^-_Q\left(s\right) 
\le \frac{\gamma^{q-p}}{|Q|},
\]
which implies that $\gamma^{2p-q}\lesssim 1$ and, since $2p-q>0$ from the assumption, $\gamma\lesssim 1$. Using this and \eqref{doublephase-sa1} with $s/\gamma$ in place of $s$, we find that 
$a_Q^+-a_Q^-\lesssim (s/\gamma)^{p-q}\lesssim s^{p-q}$, as required.

Now we prove \eqref{doublephase-sa1}. 
It follows from the H\"older continuity of $a$ that
\[
a_Q^+-a_Q^- \lesssim r^\alpha + (br^2)^{\frac\alpha2} = (1+b^\frac\alpha2)r^\alpha.
\] 
Since $s\le \phi^{-1}(z_0,|Q|^{-1})$, it suffices to check that 
\[
(1+b^\frac\alpha2)r^\alpha \lesssim \phi^{-1}(z_0,|Q|^{-1})^{p-q}.
\]
Raising both sides to the power of $\frac 1\alpha$ and using the assumption on $q-p$, we 
obtain yet another sufficient condition:
\begin{equation}\label{eq:suffDP}
(1+b^\frac12)r \lesssim \phi^{-1}(z_0,|Q|^{-1})^{\frac n{n+2}-\min\{\frac p2,1\}}.
\end{equation}

Consider first $2\le p$ so that the exponent on the right-hand side of 
\eqref{eq:suffDP} equals $-\frac2{n+2}$.
Since $\phi$ satisfies \inc{p}, 
$\lambda\mapsto \phi^{-1}(z_0,\lambda)^2$ satisfies \dec{\frac 2p}. 
Since $\frac 2p\le 1$, $b=\phi^{-1}(z_0,\lambda)^2/\lambda$ is decreasing in 
$\lambda$, $\lambda\le |Q|^{-1}$ implies 
$b \ge |Q| \, \phi^{-1}(z_0,|Q|^{-1})^2$ whereas $\lambda\ge 1$ and \azero{}
imply $1\gtrsim b$. Thus $|Q| = |B_r|\, 2r^{2}b\approx r^{n+2}b \ge r^{n+2}|Q|\, \phi^{-1}(z_0,|Q|^{-1})^2$
and so $\phi^{-1}(z_0,|Q|^{-1}) \le r^{-\frac{n+2}2}$. Hence 
\eqref{eq:suffDP} follows from 
\[
(1+b^\frac12)r\approx r =
\big(r^{-\frac{n+2}2}\big)^{-\frac 2{n+2}} \lesssim \phi^{-1}(z_0,|Q|^{-1})^{-\frac 2{n+2}}.
\]
This completes the case $2\le p$.

Consider next $2>p$ and $b\ge 1$.
From $\frac p2-\frac n{n+2}>0$, $\phi^{-1}(z_0,s)^p\le s$ and $|Q| \approx r^{n+2}b$, we obtain
\[
\phi^{-1}(z_0,|Q|^{-1})^{-p(\frac p2-\frac n{n+2})}
\gtrsim |Q|^{\frac p2-\frac n{n+2}} 
\approx (r^{n+2}b)^{\frac p2-\frac n{n+2}}
= 
r^{(\frac p2-1)n+p}b^{\frac p2-\frac n{n+2}},
\]
Using this and $b\ge 1$ in \eqref{eq:suffDP} raised to the power of $p$, 
we see that it suffices to show that 
$r^{(\frac p2-1)n+p}b^{\frac p2-\frac n{n+2}} \gtrsim b^\frac p2r^p$, 
or, equivalently, $r^{\frac p2-1}\gtrsim b^{\frac 1{n+2}}$.
From $b = \frac{\phi^{-1}(z_0, \lambda)^2}\lambda \le \lambda^{\frac 2p-1}\le |Q|^{1-\frac 2p}$ and $|Q|\approx r^{n+2} b$ 
we obtain $r^{n+2} \approx b^{-1}|Q| \le b^{\frac p{p-2}-1}=b^\frac2{p-2}$, 
which is the desired inequality.

Finally, consider $2>p$ and $b\le 1$.
We raise both sides of \eqref{eq:suffDP} to the power of $n+2$ and 
use $|Q| \approx r^{n+2}b$ to obtain the equivalent condition 
\[
|Q| \, \phi^{-1}(z_0,|Q|^{-1})^{\frac p2(n+2)-n} \lesssim b.
\]
Since $\frac{p}{2}(n+2)-n < p$, $\phi^{-1}(z_0,|Q|^{-1})\ge 
\phi^{-1}(z_0,1)\approx 1$ and $t\mapsto \phi^{-1}(z_0,t)$ satisfies \dec{1/p}, 
we estimate
\[
|Q| \, \phi^{-1}(z_0,|Q|^{-1})^{\frac p2(n+2)-n}
\lesssim
|Q| \, \phi^{-1}(z_0,|Q|^{-1})^p
\le 
\lambda^{-1} \, \phi^{-1}(z_0,\lambda)^p
\lesssim
\lambda^{-1} \, \phi^{-1}(z_0,\lambda)^2 = b. 
\qedhere
\]
\end{proof}

\begin{remark}\label{rem:borowski}
In the previous proposition, we could alternatively assume that $a$ is $C^{1, \alpha}$ in $x$ for 
$\alpha\in (0, 1]$, 
and $C^{0,\frac{1+\alpha}2}$ in $t$, in which case $\alpha$ in the upper bound in \eqref{eq:double_range} is replaced by $1+\alpha$. Thus a bigger gap can be handled in this case. 
In the proof, it is in fact sufficient to prove $a_Q^+\lesssim a_Q^- + s^{p-q}$
instead of \eqref{doublephase-sa1}. By \cite[Proposition~1.3(1)]{BorCFM24}, this holds 
by the $C^{1, \alpha}$-assumption (the $t$ direction is exactly the same).
\end{remark}

We say that $\psi: [0,\infty)\to [c_0,\infty)$ with $\psi(0)=c_0>0$, is of \textit{$\log$-type} if it is increasing and 
$\psi(s^2)\le d \psi(s)$ for some $d>1$ and all $s\ge 2$. 
This includes for instance functions like $(\log (e+s))^\alpha$, $\alpha>0$. 

\begin{proposition}[Borderline double phase]\label{prop:borderline}
Consider $\phi(z, s):= s^p + a(z) s^p \psi(s)$, where 
$p>\frac{2n}{n+2}$
and $\psi$ is an increasing function of $\log$-type. 
If $a(z)$ is a nonnegative and bounded function, and has continuity modulus $s\mapsto \psi(\frac1s)^{-1}$ in both space and time, then
$\phi$ satisfies the intrinsic parabolic \aone{} condition. 
\end{proposition}
\begin{proof}
Since $\psi$ is increasing, $\phi$ satisfies \inc{p}. Fix $q>p$. 
For $s\ge 2$ and $\lambda\ge 1$, with $k\in \mathbb N$ satisfying $\lambda s\in (s^{2^{k-1}},s^{2^{k}}]$,
\[
\frac{\phi(z,\lambda s)}{(\lambda s)^q} \le \frac{\lambda^p d^k \phi(z,s)}{ (\lambda s)^q} 
\le \frac{d^{1+\log_2(1+\log_2 \lambda)}}{\lambda^{q-p}}\frac{\phi(z,s)}{s^q} 
\le 
c\frac{ \phi(z,s)}{s^q},
\]
where $c$ depends only on $d$ from the $\log$-type assumption and $q-p$, and 
$k$ was estimated from $\lambda>s^{2^{k-1}-1}\ge 2^{2^{k-1}-1}$. 
Since further $\phi(s)\approx s^p$ for $s\in (0, 2]$, we conclude that $\phi$ 
satisfies \adec{q}.
In addition, since $\psi(2s)\le \psi(4) \le \frac{\psi(4)}{\psi(0)}\psi(s)$ if $0\le s\le 2$ and $\psi(2s)\le \psi(s^2)\le d\psi(s)$ if $s\ge2$, $\psi$ satisfies the $\Delta_2$-condition.

Let $\lambda\in [1, |Q|^{-1}]$ for $Q=Q_{r,br^2}(z_0)$ with 
$b := \frac{\phi^{-1}(z_0,\lambda)^2}\lambda$ and fix $q>p$. 
Then 
\[
\min\left\{|Q|^{1-\frac2q},1\right\}\le \lambda^{\frac2q-1}\lesssim b \le \lambda^{\frac2p-1}\le \max \left\{1, |Q|^{1-\frac2p}\right\},
\] 
hence $|Q|\approx r^{n+2}b$ implies that 
\[
\min\left\{r^{\frac{q}{2}(n+2)}, r^{n+2}\right\} \lesssim |Q|\lesssim \max\left\{r^{\frac{p}{2}(n+2)}, r^{n+2}\right\}.
\] 
Note that the first inequality above implies $r\lesssim1$ as $|Q|\le 1$,
hence $r^{\gamma_1} \lesssim |Q| \lesssim r^{\gamma_2}$ with 
$\gamma_1=\max\{\frac{q}{2}(n+2),n+2\}$ and $\gamma_2= \min\{\frac{p}{2}(n+2),n+2\}$;
a similar estimate holds for $br^2$, so that 
\[
r^{\max\{\frac{(n+2)q-2n}{2},2\}} \lesssim \min\left\{ r^2, r^{\frac{(n+2)q-2n}{2}}\right\} 
\lesssim 
b r^2 \lesssim \max\left\{ r^2, r^{\frac{(n+2)p-2n}{2}}\right\} 
\lesssim 
r^{\min\{\frac{(n+2)p-2n}{2},2\}}.
\]
The powers of $r$ above are all positive, since $q>p>\frac{2n}{n+2}$. 
Therefore, we obtain that $\frac1{r+br^2}\ge r^{-2^{-k}\gamma_1}$ for some $k\in \mathbb N$ depending only on $n,p$ and $q$.

It suffices to show that 
\[
s^p + a_Q^+s^p \psi(s) = \phi_Q^+ (s)\lesssim \phi_Q^- (s) =s^p + a_Q^-s^p\psi(s)
\quad \text{whenever } s\le |Q|^{-1}.
\]
This follows from the continuity assumption on $a$, 
the $\Delta_2$-condition on $\psi$, and the $\log$-type assumption on $\psi$: 
\[
(a_Q^+ - a_Q^-) \psi(|Q|^{-1})
\lesssim
\psi(\tfrac1{r+br^2})^{-1}\psi(r^{-\gamma_1}) 
\lesssim \psi(r^{-2^{-k}\gamma_1})^{-1}\psi(r^{-\gamma_1}) 
\le d^k. \qedhere
\]
\end{proof}

Finally, we deal with the variable exponent case. Note that the same proof also works for 
variants such as $s^{p(z)}\log (e+s)$ and $s^{p(z)}\log (e+s)^{q(z)}$.

\begin{proposition}[Variable exponent]\label{prop:varExpo}
Consider $\phi(z, s):= s^{p(z)}$ where 
$p:\Omega_T\to [p_1, p_2]$ with $\frac{2n}{n+2}<p_1\le p_2$ is $\log$-H\"older continuous. 
Then $\phi$ satisfies the intrinsic parabolic \aone{} condition. 
\end{proposition}
\begin{proof}
Let $\lambda\in [1, |Q|^{-1}]$ for $Q=Q_{r,br^2}(z_0)$ with 
$b := \frac{\phi^{-1}(z_0,\lambda)^2}\lambda$. 
Since $\phi$ satisfies \inc{p_1} and \dec{p_2},
we conclude as in the proof of Proposition~\ref{prop:borderline} that $r^{\gamma_1}\lesssim |Q|\lesssim r^{\gamma_2}$ and $r^{\gamma_1}\lesssim br^2\lesssim r^{\gamma_2}$ for some exponents $0<\gamma_2<\gamma_1$ and all $r\lesssim 1$. 
It suffices to show that 
\[
s^{p_Q^+} = \phi_Q^+ (s)\lesssim \phi_Q^- (s) =s^{p_Q^-}
\quad \text{whenever } s\in [1, |Q|^{-1}].
\]
This follows from the $\log$-Hölder continuity assumption on $p$: 
\[
s^{p_Q^+-p_Q^-}
\le 
|Q|^{-(p_Q^+-p_Q^-)}
\lesssim
\exp\Big(\frac{c\log |Q|^{-1}}{\log(e+(r+br^2)^{-1})}\Big)
\approx 1. \qedhere
\]
\end{proof}


\section{Caccioppoli inequality}\label{sect:caccioppoli}

We first introduce the parabolic and localized version of the mollification result in \cite[Theorem 4.4.7]{HarH19}.
Define 
\[
\kappa(z)=\kappa(x,t):= \tilde \kappa^n (|x|) \tilde \kappa^1(t)
\]
where $\tilde \kappa^m \in C^\infty(\R)$ is given by
\[
\tilde\kappa^m(s)=\begin{cases}
c_m \exp\left(\frac{1}{s^2-1}\right)& \text{if }\ |s|<1,\\
0 & \text{if }\ |s|\ge 1,
\end{cases} 
\]
with $c_m>0$ determined by $\int_{\R^m}\tilde \kappa^m(|x|)\,dx =1$. Denote 
$\kappa_{h}(x,t):= \frac{1}{h^{n+2}}\kappa(\frac{x}{h},\frac{t}{h^2})$
for $h> 0$. We notice that $\kappa_h \in C^\infty(\R^{n+1})$, $\kappa_h\ge 0$, $\|\kappa_h\|_{L^1(\R^{n+1})}=1$ and $\mathrm{supp}\, ( \kappa_h) \subset Q_h$. 
We further define for $f \in L^1(U)$, where $U$ is a bounded open set in $\R^{n+1}$,
\[
[f]^h(z):= (f * \kappa_{h})(z) = \int_{\R^{n+1}} f(z-w) \kappa_h(w)\,dw= \int_{Q_h} f(z-w) \kappa_h(w)\,dw, 
\]
when $Q_h(z)\subset U$. Then $[f]^h \in C^\infty(U')$ for $U' \Subset U$ and any $0<h<\min\{1,2^{-1}\mathrm{dist}(\partial U, U')\}$. Moreover, by following the same proofs of Theorem 4.4.7 and related lemmas in \cite{HarH19}, replacing the mollifiers $\sigma_\epsilon$, the balls $B$ and the usual \aone{}-condition by $\kappa_h$, the parabolic cylinders and 
the parabolic \aone{}-condition in Remark~\ref{rmk:A1para},
respectively, we can obtain the following mollification result. 
The key is that the sets in the convolution match the sets in the \aone{}-condition. 
Note that in \cite[Theorem 4.4.7]{HarH19} the mollifier is assumed to be bell-shaped. This assumption is not essential; it is only used to ensure approximation by scaled characteristic functions of balls. In our case, we work instead with characteristic functions of cylinders.

\begin{proposition}\label{mollification}
Let $U\subset \R^{n+1}$ be bounded and open, $\phi\in \Phiw(U)$ satisfy \azero{}, \adec{q} for some $q>1$, and parabolic \aone{},
and let $\kappa_h\in C^\infty_0(\R^{n+1})$ and $[f]^h$ be as above. 
Then, for $f\in L^\phi_{\loc}(U)$, $[f]^h \to f$ in $L_{\mathrm{loc}}^\phi(U)$ as $h\to 0^+$.
\end{proposition}

Now we are ready to derive a Caccioppoli inequality for the parabolic system \eqref{maineq}.

\begin{theorem}[Caccioppoli inequality]\label{lem:caccio}
Let $\phi\in \Phiw(\Omega_T)$ satisfy \azero{}, \ainc{p} and \adec{q} for some $1<p\le q$, and 
parabolic \aone{}. Let
$u$ be a weak solution to the parabolic system \eqref{maineq} with growth \eqref{A-condition}. For any $a\in\mr^N$ and $Q_{r,\rho}(z_0)\Subset \Omega_T$ with $z_0=(x_0,t_0)$, 
$\tilde r \in (0, r)$ and $\tilde\rho\in (0, \rho)$, we have 
\[
\begin{split}
\esssup_{t \in (t_0-\tilde\rho,t_0+\tilde\rho)}\int_{B_{\tilde r}} |u(x,t)-a|^2\,dx &+ \int_{Q_{\tilde r, \tilde \rho}} \phi(z,|\nabla u|)\, dz \\
&\leq c\int_{Q_{r,\rho}}\left[\frac{|u-a|^2}{\rho-\tilde\rho}+ \phi\Big(z,\Big|\frac{u-a}{r-\tilde r}\Big|\Big)\right]\, dz
\end{split}
\]
for some $c=c(n,N,p,q,L,\Lambda)>0$.
\end{theorem}

\begin{remark} The local form of parabolic \aone{} described in
Remark~\ref{rmk:A1para} suffices for Proposition~\ref{mollification} and, consequently, for Theorem~\ref{lem:caccio}.
By Remark~\ref{rmk:A1para}, intrinsic parabolic \aone{}
together with \azero{} implies
the required local parabolic \aone{} condition.
In the double phase case of Proposition~\ref{prop:DP}, one can see that parabolic \aone{} follows from the inequality
\[
q \leq p + \frac{\alpha p}{n+2}.
\]
Hence, the above Caccioppoli inequality can be obtained under this natural range for $p$ and $q$ in the parabolic setting. 
By contrast, in \cite{KimKS25} the Caccioppoli inequality was proved under the stronger inequality assumed in Proposition~\ref{prop:DP}.
\end{remark}

\begin{proof}[Proof of Theorem~\ref{lem:caccio}]
Fix $Q_{\tilde r,\tilde \rho}(z_0) \subset Q_{r,\rho}(z_0)\Subset \Omega_T$. For simplicity, we omit the center point $z_0$ from the notation.
We first observe from the weak formulation \eqref{weakform} that for every small $h>0$ 
such that $Q_h(z)\Subset \Omega_T$ for all $z\in Q_{r,\rho}$ and for every $w\in Q_h(0)$,
\[
\int_{Q_{r,\rho}}-(u(z-w)-a)\cdot \zeta_t (z) + \langle A(z-w,u(z-w),\nabla u(z-w)) , \nabla \zeta(z)\rangle \, dz=0
\]
for all $\zeta \in C^{0,1}(Q_{r,\rho})$ with $\mathrm{supp}\, \zeta \Subset Q_{r,\rho}$.
Multiplying both sides of the equation by $\kappa_h(w)$, integrating with respect to $w$, 
applying Fubini's theorem and integrating by parts, we have 
\begin{equation}\label{weakform1}
\int_{Q_{r,\rho}}([u-a]^h)_t (z) \cdot \zeta (z) + \langle [A(\cdot,u,\nabla u)]^h(z) , \nabla \zeta (z) \rangle \, dz=0.
\end{equation}
Note that $[u-a]^h \in C^\infty(Q_{r,\rho},\R^N)$ is used when integrating by parts. 
Then, since $\phi^*$ also satisfies the parabolic \aone{} (cf. Remark~\ref{rmk:phi*A1}),
we see from Proposition~\ref{mollification} that, as $h\to 0^+$,
\begin{align*}
[u-a]^h \ \ &\longrightarrow\ \ u-a 
\quad \text{in }\ L^2(Q_{r,\rho},\R^N)\cap L^\phi(Q_{r,\rho},\R^N),
\\
\nabla [u-a]^h = [\nabla u]^h \ \ &\longrightarrow\ \ \nabla u 
\quad \text{in }\ L^{\phi}(Q_{r,\rho},\R^{N\times n}),
\\
[A(\cdot,u,\nabla u)]^h \ \ &\longrightarrow\ \ A(\cdot,u,\nabla u)
\quad \text{in }\ L^{\phi^*}(Q_{r,\rho},\R^{N\times n}).
\end{align*}

Let $\eta\in C^\infty_0(Q_{r,\rho})$ be such that $0\le\eta\le1$, $\eta\equiv 1$ in $Q_{\tilde r, \tilde \rho}$, $|\nabla\eta|\le \frac{2}{r-\tilde r}$ and $|\eta_t|\leq \frac{2}{\rho-\tilde\rho}$.
Define $\tau_{t_1,\delta}\in C^{0,1}(\mr)$
for $t_1 \in [t_0-\tilde\rho, t_0+\tilde\rho]$ and $\delta \in (0,\rho-\tilde\rho)$ by 
\[
\tau_{t_1,\delta} (t)=
\begin{cases}
1 &\text{for }\ t\le t_1-\delta,\\
\frac{t_1-t}{\delta} & \text{for }\ t_1-\delta < t \le t_1,\\
0 &\text{for }\ t \ge t_1,
\end{cases}
\]
and choose 
\[
\zeta(x,t):=\eta(x,t)^{q_1}\tau_{t_1,\delta}(t) [u-a]^h(x,t),
\quad \text{where } \ q_1:=\max\{q,2\}.
\]
Equation \eqref{weakform1} holds for this $\zeta$, so that
\[\begin{split}
I_1 & :=\int_{Q_{r,\rho}}\tfrac{1}{2}\big(\big|[u-a]^h\big|^2\big)_t\eta^{q_1}\tau_{t_1,\delta} \,dz \\
& \ = - \int_{Q_{r,\rho}} \langle [A(\cdot,u,\nabla u)]^h\,,\, \eta^{q_1} \tau_{t_1,\delta} [\nabla u]^h+ q_1 \eta^{q_1-1} \tau_{t_1,\delta} [u-a]^h \nabla \eta \rangle \, dz =: -I_2.
\end{split}\]
Integration by parts and the definitions of $\eta$ and $\tau_{t_1,\delta}$ give 
\[\begin{split}
I_1 
& = - \tfrac{q_1}{2}\int_{Q_{r,\rho}}\big|[u-a]^h\big|^2\eta^{q_1-1}\eta_t \tau_{t_1,\delta} \,dz - \tfrac{1}{2}\int_{Q_{r,\rho}}\big|[u-a]^h\big|^2\eta^{q_1} (\tau_{t_1,\delta} )_t \,dz \\
 & \ge - c \int_{Q_{r,\rho}}\frac{\big|[u-a]^h\big|^2}{\rho-\tilde\rho} \,dz +\tfrac{1}{2} \fint_{t_1-\delta}^{t_1}\int_{B_{r}}\big|[u-a]^h\big|^2 \eta^{q_1} \,dx\, dt\\
& \rightarrow - c \int_{Q_{r,\rho}}\frac{|u-a|^2}{\rho-\tilde\rho} \,dz +\tfrac{1}{2} \int_{B_{r}}|u(x,t_1)-a|^2 \eta^{q_1} \,dx 
\\ &\qquad\text{as }\ h \to 0^+ \ \text{and then}\ \delta\to 0^+, 
\end{split}\]
for a.e.\ $t_1$.
By Young's inequality and \ainc{q_1'} of $\phi^*$, 
\[
\tfrac{\phi(z,|\nabla u|)}{|\nabla u|} \tfrac{ | u-a| }{r-\tilde r} \eta^{q_1-1} 
\le \phi^*(z, \tfrac{\phi(z,|\nabla u|)}{|\nabla u|} \eta^{q_1-1}) + \phi(z, \tfrac{ | u-a| }{r-\tilde r}) 
\lesssim \phi(z,|\nabla u|) \eta^{q_1} + \phi(z, \tfrac{ | u-a| }{r-\tilde r}).
\]
Using the growth conditions \eqref{A-condition} and this inequality, we find that 
\[\begin{split}
- \lim_{h\to 0^+} I_2\ 
&= 
-\int_{Q_{r,\rho}} \langle A(z,u,\nabla u)\,,\, \eta^{q_1} \tau_{t_1,\delta} \nabla u+ q_1 \eta^{q_1-1} \tau_{t_1,\delta} \, (u-a) \nabla \eta \rangle \, dz\\
& \le 
- c \int_{Q_{r,\rho}} \phi(z,|\nabla u|) \eta^{q_1} \tau_{t_1,\delta} \, dz + \int_{Q_{r,\rho}} \frac{\phi(z,|\nabla u|)}{|\nabla u|} \frac{ | u-a| }{r-\tilde r} \eta^{q_1-1} \tau_{t_1,\delta} \, dz \\
& \le 
- \tfrac{c}{2} \int_{Q_{r,\rho}} \phi(z,|\nabla u|) \eta^{q_1} \tau_{t_1,\delta} \, dz + c\int_{Q_{r,\rho}} \phi \left(z,\frac{ |u-a| }{r-\tilde r} \right) \, dz \\
& \to 
- \tfrac{c}{2} \int_{t_0-\rho}^{t_1}\int_{B_r} \phi(z,|\nabla u|) \eta^{q_1} \, dx\, dt + c\int_{Q_{r,\rho}} \phi \left(z,\frac{ |u-a| }{r-\tilde r} \right) \, dz
\quad \text{as }\ \delta\to 0^+.
\end{split}\]
Combining the above results, we obtain that for a.e.\ 
$t_1\in (t_0-\tilde\rho,t_0+\tilde\rho)$ 
\[\begin{split}
&\int_{B_{\tilde r}}|u(x,t_1)-a|^2 \,dx +\int_{t_0-\tilde \rho}^{t_1}\int_{B_{\tilde r}} \phi(z,|\nabla u|) \, dx\, dt\\
&\qquad \le c \int_{Q_{r,\rho}} \phi \left(z,\frac{ |u-a| }{r-\tilde r} \right) \, dz + c \int_{Q_{r,\rho}}\frac{|u-a|^2}{\rho-\tilde\rho} \,dz ,
\end{split}\]
since $\eta=1$ in $B_{\tilde r}$. Taking the essential supremum on the 
left-hand side, we obtain the claim.
\end{proof}


\section{Sobolev--Poincar\'e and reverse H\"older type inequalities}
\label{sect:inequalities}

In this section, we suppose that $\phi\in\Phiw(\Omega_T)$ satisfies \azero{}, \aone{}, \ainc{p} and \adec{q} with $\frac{2n}{n+2}<p\le q$ and
\begin{equation}\label{eq:newcond}
q < p+ 2\min\left\{\frac{p}{n},1\right\}.
\end{equation}
We obtain a Sobolev--Poincar\'e type inequality for the weak solution to \eqref{maineq} in the intrinsic cylinders subject to restriction in the lemma. 

{\begin{lemma}\label{lem:poincare}
Let $u$ be a weak solution to the parabolic system \eqref{maineq} with growth \eqref{A-condition}. 
If $Q^\lambda_{2r}=Q^\lambda_{2r}(z_0)\Subset \Omega_T$ 
is such that 
$1\le \lambda \le |Q^\lambda_{2r}|^{-1}$,
\[
\int_{Q^\lambda_{2r}}\phi(z,|\nabla u|)\,dz \le 1
\quad\text{and}\quad
\fint_{Q^\lambda_{2r}} \phi(z,|\nabla u|)\, dz \le c_0 \lambda
\quad
\text{for some }c_0\ge 1,
\]
then there exists $\theta=\theta(n,p,q)\in(0,1)$ such that for every $\delta\in(0,1)$, 
\[
\frac{\lambda}{\phi^{-1}(z_0,\lambda)^2}\fint_{Q^\lambda_r}\frac{|u-(u)^\lambda_r|^2}{r^2}\,dz + \fint_{Q^\lambda_r}\phi^+_{Q^\lambda_{2r}}\left(\frac{|u-(u)^\lambda_r|}{r}\right)\,dz \le c \delta \lambda + c_\delta(\Theta_0+1).
\]
where $\Theta_0:= (\fint_{Q^\lambda_{2r}} \phi(z,|\nabla u|)^\theta \, dz)^{\frac{1}{\theta}}$ and $c>0$ depends on $n$, $N$, $p$, $q$, $\Lambda$, $L$ and $c_0$, while $c_\delta>0$ additionally depends on $\delta$.
\end{lemma}}

\begin{proof}[{ Proof}]
In this proof we abbreviate $\phi^\pm(s):=\phi^{\pm}_{Q^\lambda_{2r}}(s)$. 
Since $1\le\lambda\le|Q_{2r}^\lambda|^{-1}$ and $\Theta_0\le\fint_{Q_{2r}^\lambda}\phi(z,|\nabla u|)\,dz\le|Q_{2r}^\lambda|^{-1}$, \aone{} on $Q_{2r}^\lambda$, \azero{}, and Remarks~\ref{rmk:A1} and \ref{rmk:phi*A1} imply $(\phi^{-})^{-1}(\lambda)\approx (\phi^{+})^{-1}(\lambda)$, $((\phi^{+})^*)^{-1}(\lambda)\approx ((\phi^{-})^*)^{-1}(\lambda)$ and $(\phi^{-})^{-1}(\Theta_0)\lesssim (\phi^{+})^{-1}(\Theta_0+1)$.

\textit{Step 1. (Preliminaries).}
Let $\eta\in C^\infty_0(B_r(0))$ be such that $\eta\ge 0$, $\int_{B_r(0)}\eta\, dx =1$ and $|\eta|+r |\nabla \eta|\le c(n)r^{-n}$, and define $\langle u\rangle_\eta(t):= \int_{B_r(0)}u(x,t)\eta(x)\, dx$.

For $\rho \in [r, \frac{3}{2}r]$, we define 
\[
Z(\rho):= \frac{\lambda}{\phi^{-1}(z_0,\lambda)^2}\fint_{Q^\lambda_\rho}\frac{|u-(u)^\lambda_\rho|^2}{\rho^2}\,dz + \fint_{Q^\lambda_\rho}\phi^+\bigg(\frac{|u-(u)^\lambda_\rho|}{\rho}\bigg)\,dz.
\]
Then 
\[\begin{split}
Z(\rho) & \lesssim \esssup_{t,\sigma \in I^\lambda_\rho} \left[\frac{\lambda}{\phi^{-1}(z_0,\lambda)^2} \frac{|\langle u\rangle_{\eta}(t)-\langle u\rangle_{\eta}(\sigma)|^2}{\rho^2} + \phi^+\left(\frac{|\langle u\rangle_{\eta}(t)-\langle u\rangle_{\eta}(\sigma)|}{\rho}\right)\right]\\
&\qquad+ \left[\frac{\lambda}{\phi^{-1}(z_0,\lambda)^2}\fint_{Q^\lambda_\rho}\frac{|u-\langle u\rangle_{\eta}(t)|^2}{\rho^2}\,dz + \fint_{Q^\lambda_\rho}\phi^+\left(\frac{|u-\langle u\rangle_{\eta}(t)|}{\rho}\right)\,dz\right]\\
&=: Z_1(\rho) + Z_2(\rho).
\end{split}\]

\textit{Step 2 (Estimate for the $Z_1(\rho)$-term).} We start with estimating
\[
\underset{t,\sigma\in I^\lambda_{2r}}{\mathrm{ess\, sup}\ } \frac{|\langle u\rangle_\eta(t) - \langle u\rangle_\eta(\sigma) |}{r}. 
\]
Let $t_1,t_2\in I^\lambda_{2r}$ and $t_1<t_2$. For sufficiently small $h>0$, set 
\[
\tau_{h} (t) :=
\begin{cases}
0 &\text{for }\ t\le t_1-h,\\
1+\frac{t-t_1}{h} & \text{for }\ t_1-h < t \le t_1,\\
1&\text{for }\ t_1< t \le t_2,\\
1+\frac{t_2-t}{h} & \text{for }\ t_2 < t \le t_2 +h,\\
0 &\text{for }\ t\ge t_2+h.
\end{cases}
\]
For each $i=1,\dots,N$, from the weak formulation \eqref{weakform} with test functions 
$\zeta(x,t)=\eta(x)\tau_{h} (t) e_i$ in the $i^\text{th}$ coordinate direction, we find that
\begin{align*}
|\langle u^i\rangle_\eta(t_2) - \langle u^i\rangle_\eta(t_1)|
&=\lim_{h\to0^+} \left| \fint^{t_2+h}_{t_2} \int_{B_r} u^i(x,t)\eta(x)\, dx\, dt - \fint^{t_1}_{t_1-h} \int_{B_r} u^i(x,t)\eta(x)\, dx\, dt \right|\\
&=\lim_{h\to0^+} \left| \int_{Q^\lambda_{2r}} u^i(x,t) \left[\eta(x)\tau_{h} (t)\right]_t\, dz \right| \\
&= \lim_{h\to0^+} \left|\int_{Q^\lambda_{2r}} A^i(z,u,\nabla u) \cdot \nabla\left[\eta(x)\tau_{h} (t)\right] \,dz\right|\\
&\lesssim 
r^{-n-1} \int_{Q_{2r}^\lambda} \frac{\phi (z,|\nabla u|)}{|\nabla u|}\,dz \approx r \frac{\phi^{-1}(z_0,\lambda)^2}{\lambda}\fint_{Q_{2r}^\lambda} \frac{\phi (z,|\nabla u|)}{|\nabla u|}\,dz,
\end{align*}
where $A^i$ is the $i$-th row of $A$.
We use Jensen's inequality with function $(\phi^+)^*(s)^{\theta}$ which satisfies \ainc{1} when 
$\theta\in[ \frac{q-1}{q},1)$. We note that 
$(\phi^+)^*(\tfrac{\phi(z,s)}{s})\le \phi^*(z,\tfrac{\phi(z,s)}{s})\lesssim \phi(z,s)$ and 
$\frac{\phi^{-1}(z_0,\lambda)}{\lambda} \approx \frac1{(\phi^*)^{-1}(z_0,\lambda)}\approx 
\frac1{((\phi^+)^*)^{-1}(\lambda)}=:\frac1{\psi(\lambda)}$. Thus,
\begin{align*}
J(t_1,t_2):= \frac{|\langle u\rangle_\eta(t_1) - \langle u\rangle_\eta(t_2) |}{r} 
\lesssim 
\frac{\phi^{-1}(z_0,\lambda)}{\psi(\lambda)}
\psi\bigg(\underbrace{\left[\fint_{Q_{2r}^\lambda} \phi(z,|\nabla u|)^{\theta}\,dz\right]^{\frac{1}{\theta}}}_{=\Theta_0}\bigg).
\end{align*}
We consider two cases. If $\Theta_0\le \delta \lambda$, then 
\[
J(t_1,t_2)
\le 
\frac{\phi^{-1}(z_0,\lambda)}{\psi(\lambda)}\psi(\delta\lambda) 
\lesssim
\delta^{\frac1{p'}} \phi^{-1}(z_0,\lambda).
\]
In the opposite case, $\lambda<\delta^{-1}\Theta_0$. 
Furthermore, by assumption and Hölder's inequality, $\Theta_0\le c_0\lambda$. 
By \adec{\theta p'} of $\psi$ and \adec{\frac1p} of $\phi^{-1}$, 
\[
J(t_1,t_2)
\le
\frac{\phi^{-1}(z_0,\delta^{-1}\Theta_0)}{\psi(\lambda)}\psi(c_0 \lambda) 
\lesssim
\delta^{-\frac1{p}} \phi^{-1}\left(z_0, \Theta_0 \right).
\]
Combining the two cases and applying \aone{}, we obtain that for every $\delta\in(0,1)$,
\begin{equation}\label{eq:J1}
\begin{split}
\underset{t,\sigma \in I^\lambda_{2r}}{\mathrm{ess\,sup}\ }\bigg| \frac{\langle u\rangle_{\eta}(t)-\langle u\rangle_{\eta}(\sigma)}r \bigg| 
&\lesssim
\delta^{\frac1{p'}} \phi^{-1}(z_0,\lambda) + \delta^{-\frac1{p}}\phi^{-1}(z_0,\Theta_0)\\
&\lesssim \delta^{\frac1{p'}} (\phi^+)^{-1}(\lambda) + \delta^{-\frac1{p}}(\phi^+)^{-1}(\Theta_0+1).
\end{split}\end{equation}

Hence, by the first inequality of \eqref{eq:J1} and \aone{} we find that 
\[
Z_1(\rho) \lesssim \delta^{\frac2{p'}} \lambda + \delta^{-\frac2{p}} \frac{\lambda \phi^{-1}(z_0,\Theta_0)^2}{\phi^{-1}(z_0,\lambda)^2} +\delta^{p-1} \lambda + \delta^{-\frac{q}{p}} \big(\Theta_0+1\big).
\]
We estimate the second term by Young's inequality with $\phi^*(z_0,\frac{\lambda }{\phi^{-1}(z_0,\lambda)})\approx \lambda$,
\[
\delta^{-\frac2{p}} \frac{\lambda \phi^{-1}(z_0,\Theta_0)^2}{\phi^{-1}(z_0,\lambda)^2} 
\lesssim 
\delta^{-\frac2{p}} \frac{\lambda \phi^{-1}(z_0,\Theta_0)}{\phi^{-1}(z_0,\lambda)} 
\lesssim 
\delta^{-\frac2{p}}(\phi^*)^{-1}(z_0, \lambda) \phi^{-1}(z_0, \Theta_0) 
\lesssim 
\delta \lambda + \delta^{-\alpha_0} \Theta_0.
\]
for some $\alpha_0>0$. Continuing the estimate of $Z_1$, we obtain that 
\begin{equation}\label{eq:estimateZ1}\begin{split}
Z_1(\rho) 
&\lesssim 
\delta^{\frac2{p'}} \lambda + \delta \lambda + \delta^{-\alpha_0} \Theta_0+ \delta^{p-1} \lambda + \delta^{-\frac{q}{p}} \big(\Theta_0+1\big)\\
& \lesssim 
\delta^{\min\{\frac2{p'}, p-1,1\}} \lambda + \delta^{-\alpha_0} \big(\Theta_0+1\big).
\end{split}\end{equation}

\textit{Step 3 (Estimate for the $Z_2(\rho)$-term).} 
Let $r \le {\tilde\rho} <\rho \le \frac{3}{2}r$. We start with recalling the following Gagliardo--Nirenberg inequality (see \cite{Nirenberg}): for $p_1\ge 1$,$\gamma_1,q_1>0$ and ${\theta_1}\in[0,1]$, 
\begin{equation}\label{eq:GN}
\bigg(\fint_{B_{r}}\big|\tfrac{f}{r}\big|^{\gamma_1}\, dx\bigg)^{\frac1\gamma_1}
\leq c
\bigg(\fint_{B_r}|\nabla f|^{p_1}\,dx\bigg)^{\frac{{\theta_1}}{p_1}}\bigg(\fint_{B_r}\big|\tfrac{f}{r}\big|^{q_1}\,dx\bigg)^{\frac{1-{\theta_1}}{q_1}}
\end{equation}
for some $c=c(n,p_1,\gamma_1,q_1,{\theta_1})>0$, provided that
\begin{equation}\label{eq:GNcond}
\frac{1}{\gamma_1}\geq {\theta_1}\left(\frac{1}{p_1}-\frac{1}{n}\right)+\frac{1-{\theta_1}}{q_1}
\end{equation}
and the $p_1$-Poincar\'e inequality $\int_{B_r}\big|\tfrac{f}{r}\big|^{p_1}\,dx 
\lesssim \int_{B_r}|\nabla f|^{p_1}\,dx$ holds. In the application below, the Poincar\'e inequality
follows from $\int_{B_r}f\eta\,dx=0$, where
$\eta\ge0$, $\int_{B_r}\eta\,dx=1$, and
$\|\eta\|_\infty\le c(n)r^{-n}$.

For the second term of $Z_2({\tilde\rho})$, applying the Jensen inequality in Lemma~\ref{lem:Jensen} to the 
functions $t\mapsto (\phi^+)^{-1}(t)^q$ and $t\mapsto \phi^-(t^{\frac{1}{p\theta}})^{\theta}$, the 
Gagliardo--Nirenberg inequality \eqref{eq:GN} to $f=u-\langle u \rangle_\eta(t)$ with $(\gamma_1,p_1,q_1,{\theta_1})=(q,p\theta,2,\frac{p\theta}{q})$, where $\theta\in(0,1)$, we find that 
\[
\begin{split}
& \fint_{Q_{\tilde\rho}^\lambda}\phi^+\bigg(\bigg|\frac{u(z)-\langle u\rangle_{\eta}(t)}{\tilde\rho}\bigg|\bigg) \,dz 
 \lesssim
 \phi^+\left( \bigg[\fint_{Q_{\tilde\rho}^\lambda} \bigg|\frac{u(z)-\langle u\rangle_{\eta}(t)}{\tilde\rho}\bigg|^q\,dz \bigg]^\frac1q\right)\\
&\quad\lesssim 
\phi^+\left( \bigg[ \fint_{I_{\tilde\rho}^\lambda} \bigg( \fint_{B_{\tilde\rho}}|\nabla u|^{p\theta} \,dx\bigg) \bigg( \fint_{B_{\tilde\rho}} \bigg|\frac{u(z)-\langle u\rangle_{\eta}(t)}{\tilde\rho}\bigg|^2\,dx\bigg)^{\frac{q(1-{\theta_1})}{2}} \,dt\bigg]^{\frac{1}{q}}\right) \\
&\quad\lesssim 
\phi^+\left(\bigg[ \bigg( \fint_{Q_{\tilde\rho}^\lambda} |\nabla u|^{p\theta} \,dz \bigg) \bigg( \underset{t\in I^\lambda_{\tilde\rho}}{\mathrm{ess\, sup}\ } \fint_{B_{\tilde\rho}} \bigg|\frac{u(z)-\langle u\rangle_{\eta}(t)}{\tilde\rho}\bigg|^2\, dx \bigg)^{\frac{q(1-{\theta_1})}{2}} \bigg]^{\frac{1}{q}}\right) \\
&\quad\lesssim 
\phi^+\left( (\phi^-)^{-1}\bigg(\bigg[\fint_{Q_{\tilde\rho}^\lambda}\phi^- (|\nabla u|)^{\theta} \,dz\bigg]^{\frac{1}{\theta}}\bigg)^{{\theta_1}} \bigg[ \underset{t\in I^\lambda_{\tilde\rho}}{\mathrm{ess\, sup}\ } \fint_{B_{\tilde\rho}} \bigg|\frac{u(z)-\langle u\rangle_{\eta}(t)}{\tilde\rho}\bigg|^2\, dx \bigg]^{\frac{1-{\theta_1}}{2}}\right);
\end{split}
\]
with these parameters we see that \eqref{eq:GNcond} is equivalent to 
$q\le\theta p(1+\frac 2n)$, which holds for $\theta$ sufficiently close to $1$ by \eqref{eq:newcond}.
By \aone{}, 
\begin{equation}\label{eq:estimateZ21}
(\phi^-)^{-1}\bigg(\bigg[\fint_{Q_{\tilde\rho}^\lambda}\phi^- (|\nabla u|)^{\theta} \,dz\bigg]^{\frac{1}{\theta}}\bigg) 
 \lesssim (\phi^+)^{-1}(\Theta_0 +1).
\end{equation}
By the Caccioppoli inequality in Theorem~\ref{lem:caccio} with $a=(u)^\lambda_\rho$ and $(\tilde r,r,\tilde\rho,\rho)$ replaced by $(\tilde\rho,\rho,b\tilde\rho^2,b\rho^2)$, where $b:=\phi^{-1}(z_0,\lambda)^2/\lambda$, and by the properties of $\eta$, with $\tilde q:=\max\{2,q\}$,
\begin{equation}\label{eq:estimateZ22}
\begin{aligned}
&\esssup_{t\in I^\lambda_{\tilde\rho}}
\fint_{B_{\tilde\rho}}\left|\frac{u(z)-\langle u\rangle_\eta(t)}{\tilde\rho}\right|^2\,dx\\
&\quad\lesssim
\esssup_{t\in I^\lambda_{\tilde\rho}}
\fint_{B_{\tilde\rho}}\left|\frac{u(z)-(u)^\lambda_\rho}{\tilde\rho}\right|^2\,dx\\
&\quad\lesssim\fint_{Q^\lambda_\rho}\frac{|u-(u)^\lambda_\rho|^2}{\rho^2-\tilde\rho^2}\,dz
+\frac{(\phi^+)^{-1}(\lambda)^2}{\lambda}
\fint_{Q^\lambda_\rho}\phi^+\bigg(\frac{|u-(u)^\lambda_\rho|}{\rho-\tilde\rho}\bigg)\,dz\\
&\quad\lesssim\left(\frac{r}{\rho-\tilde\rho}\right)^{\tilde{q}}
\frac{(\phi^+)^{-1}(\lambda)^2}{\lambda}Z(\rho)+(\phi^+)^{-1}(\lambda)^2.
\end{aligned}\end{equation}
Combining the last three estimates and setting 
\begin{equation}\label{eq:tilde}
\tilde \lambda := (\phi^+)^{-1}(\lambda)
\quad\text{and}\quad
\tilde \Theta_0 := (\phi^+)^{-1}(\Theta_0+1),
\end{equation}
we get
\begin{equation}\label{eq:estimateZ2second}\begin{aligned}
&\fint_{Q_{\tilde\rho}^\lambda}\phi^+\left(\left|\frac{u(z)-\langle u\rangle_\eta(t)}{\tilde\rho}\right|\right)\,dz\\
&\quad\lesssim\phi^+\Bigg(\left(\frac{r}{\rho-\tilde\rho}\right)^{\tilde{q}(1-\theta_1)/2}
\tilde\Theta_0^{\theta_1}\bigg[\frac{\tilde\lambda^2}{\phi^+(\tilde\lambda)}Z(\rho)\bigg]^{(1-\theta_1)/2}
+\tilde\Theta_0^{\theta_1}\tilde\lambda^{1-\theta_1}\Bigg).
\end{aligned}\end{equation}
Since $q<p+2$ by \eqref{eq:newcond}, we can choose $\theta_1$ so close to $1$ that
$\kappa:=\frac12(1-\theta_1)$ satisfies $\tilde{q}\kappa<1$.
Write $R:=r/(\rho-\tilde\rho)$, $X:=(\phi^+)^{-1}(Z(\rho))$ and
$\tilde\delta_1:=R^{-\tilde{q}\kappa/(1-\tilde{q}\kappa)}\tilde\delta$ for $\tilde\delta\in(0,1)$.
Using $\tilde\Theta_0\lesssim\tilde\lambda$, and considering separately
$\tilde\lambda\le\tilde\delta_1X$ (with $s\mapsto \frac{s}{\phi^+(s)^\kappa}$ almost increasing) and $\tilde\lambda>\tilde\delta_1X$, we obtain
\[
\begin{split}
\tilde\Theta_0^{\theta_1}\bigg[\frac{\tilde\lambda^2}{\phi^+(\tilde\lambda)}Z(\rho)\bigg]^\kappa
&\lesssim
\tilde\Theta_0^{1-\tilde{q}\kappa}
\bigg[\frac{\tilde\lambda^{\tilde{q}}}{\phi^+(\tilde\lambda)}\bigg]^\kappa Z(\rho)^\kappa \\
&\lesssim
\frac{\tilde\lambda}{\phi^+(\tilde\lambda)^\kappa} \phi^+(X)^\kappa
+
\tilde\Theta_0^{1-\tilde{q}\kappa}
\bigg[\frac{\tilde\lambda^{\tilde{q}}}{\phi^+(\tilde\lambda)}\bigg]^\kappa \phi^+(\tilde\delta_1^{-1}\tilde\lambda)^\kappa \\
&\lesssim
\tilde\delta_1^{1-\tilde{q}\kappa}X
+\tilde\delta_1^{-\tilde{q}\kappa}\tilde\Theta_0^{1-\tilde{q}\kappa}\tilde\lambda^{\tilde{q}\kappa}.
\end{split}
\]
Inserting this into \eqref{eq:estimateZ2second}, and using
$R^{\tilde{q}\kappa}\tilde\delta_1^{-\tilde{q}\kappa}
=R^{\tilde{q}\kappa/(1-\tilde{q}\kappa)}\tilde\delta^{-\tilde{q}\kappa}$, gives
\begin{equation}\label{eq:estimateZ2second-final}
\begin{split}
&\fint_{Q_{\tilde\rho}^\lambda}\phi^+\left(\left|\frac{u(z)-\langle u\rangle_\eta(t)}{\tilde\rho}\right|\right)\,dz\\
&\quad\lesssim\tilde\delta^{p(1-\tilde{q}\kappa)}Z(\rho)
+\phi^+(\tilde\Theta_0^{\theta_1}\tilde\lambda^{1-\theta_1})
+R^{\alpha_1}\tilde\delta^{-\alpha_2}\phi^+(\tilde\Theta_0^{1-\tilde{q}\kappa}\tilde\lambda^{\tilde{q}\kappa}),
\end{split}
\end{equation}
where $\alpha_1:=q\tilde{q}\kappa/(1-\tilde{q}\kappa)$ and $\alpha_2:=q\tilde{q}\kappa$. 
We next estimate the first term of $Z_2({\tilde\rho})$. Applying the inequality \eqref{eq:GN} to $f=u-\langle u \rangle_\eta(t)$ with $(\gamma_1,p_1,q_1)$ by $(2,\frac{2n}{n+2},2)$ and ${\theta_1}={\theta_2} := \min\{\frac{p\theta}{2},\frac12\}\in(0,1)$ and applying the Jensen inequality in Lemma~\ref{lem:Jensen} to the function $t\mapsto\phi^-(t^{1/(p\theta)})^\theta$, we find that
\begin{equation}\label{eq:L2interpolation}
\begin{split}
& \fint_{Q_{\tilde\rho}^\lambda}\bigg|\frac{u(z)-\langle u\rangle_{\eta}(t)}{\tilde\rho}\bigg|^2 \,dz \\
& \quad\lesssim
 \fint_{I_{\tilde\rho}^\lambda} \bigg( \fint_{B_{\tilde\rho}}|\nabla u|^{\frac{2n}{n+2}} \,dx\bigg)^{\frac{{\theta_2}(n+2)}{n}} \bigg( \fint_{B_{\tilde\rho}} \bigg|\frac{u(z)-\langle u\rangle_{\eta}(t)}{\tilde\rho}\bigg|^2\,dx\bigg)^{1-{\theta_2}} \,dt\\
&\quad\lesssim
 \fint_{I_{\tilde\rho}^\lambda} \bigg( \fint_{B_{\tilde\rho}}|\nabla u|^{p\theta} \,dx\bigg)^{\frac{2{\theta_2}}{p\theta}} \bigg( \fint_{B_{\tilde\rho}} \bigg|\frac{u(z)-\langle u\rangle_{\eta}(t)}{\tilde\rho}\bigg|^2\,dx\bigg)^{1-{\theta_2}} \,dt\\ 
&\quad\lesssim 
 \bigg( \fint_{Q_{\tilde\rho}^\lambda} |\nabla u|^{p\theta} \,dz \bigg)^{\frac{2{\theta_2}}{p\theta}} \bigg( \underset{t\in I^\lambda_{\tilde\rho}}{\mathrm{ess\, sup}} \fint_{B_{\tilde\rho}} \bigg|\frac{u(z)-\langle u\rangle_{\eta}(t)}{\tilde\rho}\bigg|^2\, dx \bigg)^{1-{\theta_2}}\\
&\quad\lesssim (\phi^-)^{-1}\bigg(\bigg[\fint_{Q_{\tilde\rho}^\lambda}\phi^- (|\nabla u|)^{\theta} \,dz\bigg]^{\frac{1}{\theta}}\bigg)^{2{\theta_2}} \bigg[\underset{t\in I^\lambda_{\tilde\rho}}{\mathrm{ess\, sup}} \fint_{B_{\tilde\rho}} \bigg|\frac{u(z)-\langle u\rangle_{\eta}(t)}{\tilde\rho}\bigg|^2\, dx \bigg]^{1-{\theta_2}}, 
\end{split}
\end{equation}
where we further assume that $\theta\in(0,1)$ is sufficiently close to $1$, so that $p\theta >\frac{2n}{n+2}$. We then recall \eqref{eq:estimateZ21}, \eqref{eq:estimateZ22}, \eqref{eq:tilde} and the inequality $\tilde\Theta_0\lesssim \tilde\lambda$, to obtain
\begin{align}\label{eq:estimateZ2first}
\notag\frac{\lambda}{\phi^{-1}(z_0,\lambda)^2} \fint_{Q_{\tilde\rho}^\lambda}\bigg|\frac{u(z)-\langle u\rangle_{\eta}(t)}{\tilde\rho}\bigg|^2 \,dz 
& \lesssim \frac{\phi^+(\tilde\lambda)}{\tilde\lambda^2} \tilde\Theta_0^{2{\theta_2}} \bigg[ \Big(\frac{r}{\rho-{\tilde\rho}}\Big)^{\tilde{q}}\frac{\tilde\lambda^2}{(\phi^+)(\tilde\lambda)} Z(\rho) + \tilde\lambda^2\bigg]^{1-{\theta_2}} \\
& \lesssim \big(\tfrac{r}{\rho-{\tilde\rho}}\big)^{\tilde{q}(1-{\theta_2})} \tilde \Theta_0^{{\theta_2}} \Big(\frac{\phi^+(\tilde\lambda)}{\tilde\lambda}\Big)^{{\theta_2}} Z(\rho)^{1-{\theta_2}} + \frac{\phi^+(\tilde\lambda)}{\tilde\lambda^{{\theta_2}}} \tilde\Theta_0^{{\theta_2}} \\
\notag& \lesssim 
\tilde \delta Z(\rho) + \tilde\delta^{-\frac{1-{\theta_2}}{{\theta_2}}} \big(\tfrac{r}{\rho-{\tilde\rho}}\big)^{\frac{\tilde{q}(1-{\theta_2})}{{\theta_2}}} \tilde \Theta_0 \frac{\phi^+(\tilde\lambda)}{\tilde\lambda} + \frac{\phi^+(\tilde\lambda)}{\tilde\lambda^{{\theta_2}}} \tilde\Theta_0^{{\theta_2}}.
\end{align}

\textit{Step 4 (Conclusion).}
The iteration below is applicable since $Z$ is bounded on $[r,3r/2]$: Indeed, \eqref{eq:GN} with $(\gamma_1,p_1,q_1,\theta_1)=(q,p,2,p/q)$, the energy assumption and $u\in L^\infty(I_{2r}^\lambda;L^2(B_{2r}))$ imply $u-\langle u\rangle_\eta\in L^q(Q_{3r/2}^\lambda)$. Together with \eqref{eq:J1}, \azero{} and \adec{q}, this gives the required boundedness.
Combining the estimates $Z\lesssim Z_1+Z_2$, \eqref{eq:estimateZ1}, \eqref{eq:estimateZ2second-final} and \eqref{eq:estimateZ2first}, we arrive at, for every $\delta, \tilde\delta\in(0,1)$ and $r\le {\tilde\rho} < \rho \le \frac32r$,
\[\begin{split}
Z(\tilde\rho) 
& \lesssim 
\tilde\delta Z(\rho) + \big(\tfrac{r}{\rho-{\tilde\rho}}\big)^{\alpha_3} \tilde \delta^{-\alpha_4} \bigg\{ \phi^+\Big(\tilde{\Theta}_0^{1-\max\{q,2\}\frac{1-{\theta_1}}{2}} \tilde{\lambda}^{\max\{q,2\}\frac{1-{\theta_1}}{2}}\Big)+ \tilde\Theta_0\frac{\phi^+(\tilde\lambda)}{\tilde\lambda}\bigg\}\\
&\qquad +\delta^{\min\{\frac2{p'}, p-1,1\}}\lambda + \delta^{-\alpha_0} \big(\Theta_0+1\big)
 + \frac{\phi^+(\tilde\lambda)}{\tilde\lambda^{{\theta_2}}} \tilde\Theta_0^{{\theta_2}} + \phi^+ (\tilde\Theta_0^{{\theta_1}}\tilde\lambda^{1-{\theta_1}})
\end{split}\]
for some $\alpha_3,\alpha_4>0$.
By a standard iteration argument (see for instance \cite[Lemma 4.2]{HarHT17}) with $\tilde \delta$ sufficiently small, and considering the 
two cases $\tilde\Theta_0\le\delta\tilde\lambda$ and $\tilde\Theta_0>\delta\tilde\lambda$ with \eqref{eq:tilde}, we conclude that 
\[\begin{split}
Z(r) 
&\lesssim 
\phi^+\left(\tilde{\Theta}_0^{1-\max\{q,2\}\frac{1-{\theta_1}}{2}} \tilde{\lambda}^{\max\{q,2\}\frac{1-{\theta_1}}{2}}\right)+ \tilde\Theta_0\Big(\frac{\phi^+(\tilde\lambda)}{\tilde\lambda}\Big) \\
&\qquad +\delta^{\min\{\frac2{p'}, p-1,1\}} \lambda + \delta^{-\alpha_0} \big(\Theta_0+1\big)
 + \frac{\phi^+(\tilde\lambda)}{\tilde\lambda^{{\theta_2}}} \tilde\Theta_0^{{\theta_2}} + \phi^+ (\tilde\Theta_0^{{\theta_1}}\tilde\lambda^{1-{\theta_1}})\\
& \lesssim 
\delta^{\alpha_5}\big(\phi^+(\tilde \lambda) + \lambda\big) + \delta^{-\alpha_6}\big( \phi^+(\tilde \Theta_0) + \Theta_0+1\big) \approx \delta^{\alpha_5} \lambda + \delta^{-\alpha_6} ( \Theta_0+1)
\end{split}\]
for some $\alpha_5,\alpha_6>0$. Consequently, since $\delta\in(0,1)$ is arbitrary, we prove the claim.
\end{proof}

{\begin{remark}\label{rmk:poincare}
In Lemma~\ref{lem:poincare}, if $\phi(z,s)$ is independent of $z$, then the proof can be modified by using the Gagliardo-Nirenberg inequality for Orlicz functions from \cite[Lemma 2.13]{HasO21} instead of the standard one \eqref{eq:GN} together with a Jensen-type inequality. Consequently, assumption \eqref{eq:newcond} is not needed. Since this is rather technical and the corresponding higher integrability result has already been established in \cite{HasO21}, we omit the details.
\end{remark}}

Next we derive a reverse H\"older inequality.

{
\begin{lemma}\label{lem:reverse}
Let $u$ be a weak solution to the parabolic system \eqref{maineq} with growth \eqref{A-condition}. Suppose that 
$Q^\lambda_{4\rho}\Subset\Omega_T$ 
satisfies 
\[
1\le\lambda
\le \min\left\{
\fint_{Q_\rho^\lambda}\phi(z,|\nabla u|)\,dz+1,\,
|Q_{4\rho}^\lambda|^{-1}
\right\} 
\]
and
\[
\fint_{Q_{4\rho}^\lambda}\phi(z,|\nabla u|)\,dz
\le \min\left\{
4^{n+2}\lambda,\,
|Q_{4\rho}^\lambda|^{-1}
\right\}.
\]
Then there exist $\theta=\theta(n,p,q)\in(0,1)$ and $c=c(n,N,p,q,\Lambda,L)>0$ such that
\[
\fint_{Q^\lambda_\rho}
\phi(z,|\nabla u|)\,dz\leq c\bigg(\fint_{Q^\lambda_{4\rho}}\left[\phi(z,|\nabla u|)+1\right]^{\theta}\,dz\bigg)^{\frac{1}{\theta}}.
\]
\end{lemma}}
\begin{proof}[{Proof}] 
Let $\theta\in(0,1)$ be from Lemma~\ref{lem:poincare} and set 
$\Theta:= (\fint_{Q^\lambda_{4\rho}}[\phi(z,|\nabla u|)+1]^{\theta}\,dz)^{\frac{1}{\theta}}$.
By Lemma~\ref{lem:poincare} with $r=2\rho$ and $c_0=4^{n+2}$,
\[ 
\frac{\lambda}{\phi^{-1}(z_0,\lambda)^2}\fint_{Q^\lambda_{2\rho}}\frac{|u-(u)^\lambda_{2\rho}|^2}{(2\rho)^2}\,dz + \fint_{Q^\lambda_{2\rho}}\phi^+\left(\frac{|u-(u)^\lambda_{2\rho}|}{2\rho}\right)\,dz 
\le c \delta \lambda + c_\delta \Theta.
\]
By the Caccioppoli inequality from Theorem~\ref{lem:caccio} with $a:=(u)_{2\rho}^\lambda$ 
and these inequalities, we find that 
\begin{align*}
\fint_{Q_\rho^\lambda} \phi(z,|\nabla u|)\, dz 
&\le c\frac{\lambda}{\phi^{-1}(z_0,\lambda)^2} \fint_{Q^\lambda_{2\rho}}\bigg|\frac{u-(u)_{2\rho}^\lambda}{2\rho}\bigg|^2\,dz 
+ c\fint_{Q^\lambda_{2\rho}} \phi\bigg(z, \bigg|\frac{u-(u)_{2\rho}^\lambda}{2\rho}\bigg|\bigg)\, dz \\
&\le c \delta \lambda +c_\delta \Theta .
\end{align*}
By assumption, the left-hand side is at least $\lambda-1$. 
With the choice $\delta=\frac1{2c}$ the term $c \delta \lambda$ can be absorbed 
in the left-hand side and we obtain the claim.
\end{proof}

\section{Proof of higher integrability}\label{sect:proof}

We start with a Vitali-type covering lemma for intrinsic parabolic cylinders. 

\begin{lemma}\label{lem:vitali}
Assume that $\phi\in \Phiw(\Omega_T)$ satisfies \azero{} and \aone{}, with constant 
$\tilde\beta\in (0,1)$ from Remark~\ref{rmk:A1}.
Let $\lambda\ge 1$, and $\mathbf F=\{Q_\alpha\}$ be a family of intrinsic parabolic cylinders 
\[
Q_\alpha= Q^\lambda_{r_\alpha}(w_\alpha) 
\quad\text{satisfying }\ \lambda \le |Q_\alpha|^{-1} \ \ \text{and}\ \ r_\alpha \le R,
\]
for some $R>0$. Then there exists a countable disjoint subcollection 
$\mathbf G$ of $\mathbf F$ such that 
\[
\bigcup_{Q\in \mathbf F} Q \, \subset \, \bigcup_{Q \in \mathbf G} (8\tilde\beta^{-4}+1) Q,
\]
where $\gamma Q_\alpha := Q^\lambda_{\gamma r_\alpha}(w_\alpha) $ for $\gamma>0$.
\end{lemma}

\begin{proof}
For $i\in \mathbb N$, let $\mathbf F_i=\{Q_\alpha\in \mathbf F: 2^{-i}R< r_\alpha \le 2^{-i+1}R \}$, and define $\mathbf G_i\subset \mathbf F_i$ inductively as follows. First, let $\mathbf G_1$ be any maximal disjoint subcollection of $\mathbf F_1=\mathbf H_1$. When $\mathbf G_{1},\dots,\mathbf G_{i}$ 
have been chosen, let $\mathbf G_{i+1}$ be any maximal disjoint subcollection of 
\[
\mathbf H_{i+1}:=\{Q \in \mathbf F_{i+1}: Q\cap Q'=\emptyset \ \ \text{for all }\ Q'\in \mathbf G_1 \cup \cdots \cup \mathbf G_i\}.
\]
Note that $\mathbf G_i$ is countable. We show that 
$\mathbf G:= \cup_{i=1}^\infty \mathbf G_i$ satisfies the requirement of 
the lemma. Indeed, if $Q\in \mathbf F$, then $Q\in \mathbf F_i$ for some $i$. 
Either $Q$ does not belong to $\mathbf H_i$, which means $i > 1$ and $Q$ 
intersects some cylinder in $\mathbf G_1\cup\cdots\cup \mathbf G_{i-1}$, or 
$Q\in \mathbf H_{i}$ and by the maximality of $\mathbf G_i$, $Q$ 
intersects some cylinder in $\mathbf G_i$. In any case, $Q$ intersects some 
$Q'\in \mathbf G_1 \cup \cdots \cup \mathbf G_i$. 
Let $Q=Q^\lambda_{r_Q}(w)$ and $Q'=Q^\lambda_{r_{Q'}}(w')$ with $w'=(x',t')$.
Note that $r_{Q}\le 2^{-i+1}R$ and $r_{Q'} > 2^{-i}R$, hence 
$2r_{Q'} \ge r_Q$. 
Then for every $(x,t)\in Q$,
\[
|x-x'| < 2r_Q + r_{Q'}\le 5r_{Q'} 
\] 
and, using the \aone{}-condition of $\phi$,
\[
|t-t'| < 
2 r_Q^2\frac{\phi^{-1}(w,\lambda)^2}\lambda +
r_{Q'}^2\frac{\phi^{-1}(w',\lambda)^2}\lambda 
\le 
(8\tilde\beta^{-4}+1) r_{Q'}^2\frac{\phi^{-1}(w',\lambda)^2}\lambda;
\]
in the last inequality, we used \aone{} to get $\phi^{-1}(w,\lambda)\le \tilde\beta^{-1}\phi^{-1}(\tilde w,\lambda)\le \tilde\beta^{-2}\phi^{-1}(w',\lambda)$, where $\tilde w \in Q\cap Q'$, see Remark~\ref{rmk:A1}. Therefore the claim holds.
\end{proof}

\begin{proof}[Proof of Theorem~\ref{thm:main}] 
We prove the main theorem in four steps.

\smallskip
\textit{Step 1 (Setting).} 
In view of Remark~\ref{rmkphiphic}, we can assume without loss of generality that $\phi$ is continuous, strictly increasing and satisfies \inc{p} and \dec{\tilde q}. Fix $Q_{2r}\Subset \Omega_T$ such that
\[
\int_{Q_{2r}} [\phi(z,|\nabla u|) +1] \,dz \le 1.
\]
Recall that $\mathcal{D}$ was defined in the theorem as 
\[
\mathcal{D}(z,s)=\min\left\{\phi^{-1}(z,s)^2,s^{\frac{n}{2}+1}\phi^{-1}(z,s)^{-n}\right\}.
\]
By \inc{p} and \dec{\tilde q} of $\phi$, $\mathcal D$ satisfies \inc{p_0} with $p_0=\min\{\frac2{\tilde q},\frac{n+2}{2}-\frac{n}{p} \}>0$, and for $s>0$ and $C\ge 1$,
\[
\min\Big\{C^{\frac{2}{\tilde q}}, C^{\frac{n+2}{2} - \frac{n}{p}}\Big\} \mathcal{D}(z,s)
=
C^{\min\{\frac{2}{\tilde q},\frac{n+2}{2} - \frac{n}{p}\}} \mathcal{D}(z,s)
\le 
\mathcal{D}(z,Cs) .
\]

Define
\[ 
\lambda_0:= 
(\mathcal{D}^-)^{-1}\bigg(\fint_{Q_{2r}}[\phi(z, |\nabla u|)+ 1]\,dz \bigg),
\quad\text{where }\ \mathcal{D}^-:=\mathcal{D}^-_{Q_{2r}}.
\]
Note that $\mathcal{D}^-$ satisfies \inc{p_0}, so its inverse is well-defined and 
$\lambda_0\ge (\mathcal{D}^-)^{-1}(1)\approx 1$ by \azero{}.
For $1\le \nu\leq 2$ and $\lambda>0$, we denote super-level sets
\[
E(\nu,\lambda)
:=
\big\{z\in Q_{\nu r}: \phi(z, |\nabla u(z)|)+1>\lambda\big\}.
\]

We next fix $1\leq \nu_1<\nu_2\leq 2$ and $\lambda$ satisfying 
\begin{equation}\label{lambda1}
\lambda \geq \lambda_1:= \Big(\frac{8\chi}{\nu_2-\nu_1}\Big)^
{\frac{n+2}{p_0} }\max\{\lambda_0,1\},
\end{equation}
with $\chi:=8\tilde\beta^{-4}+1$. 
With this $\lambda$ and $w\in Q_{2r}$, let
\[
r_{\lambda,w}:= 
\min\left\{1,\bigg(\frac{\lambda}{\phi^{-1}(w,\lambda)^2}\bigg)^{\frac12}\right\}(\nu_2-\nu_1)r \le \min\left\{1,\bigg(\frac{\lambda}{(\phi^+_{Q_{2r}})^{-1}(\lambda)^2}\bigg)^{\frac12}\right\}(\nu_2-\nu_1)r.
\]
We notice that $Q^\lambda_\rho(w)\subset Q_{\nu_2r}$ 
whenever $w\in Q_{\nu_1 r}$ and $\rho\leq r_{\lambda,w}$.

\textit{Step 2 (Covering of super-level sets).}
Set $K:=4^{n+2}$.
We prove a Vitali type covering 
of the super-level set $E(\nu_1,K\lambda)$ satisfying a balancing condition on each set. 

For $w\in E(\nu_1,K\lambda)$ and $\rho\in[\frac{r_{\lambda,w}}{4\chi},r_{\lambda,w})$, using 
the definition of $\lambda_0$, we have
\[\begin{split}
\fint_{Q^\lambda_{\rho}(w)}[\phi(z, |\nabla u|)+ 1]\,dz 
\leq \frac{|Q_{2r}|}{|Q^\lambda_{\rho}(w)|} \fint_{Q_{2r}}[\phi(z, |\nabla u|)+ 1]\,dz 
\leq \Big(\frac{2r}{\rho}\Big)^{n+2}\frac{\lambda \mathcal{D}^-(\lambda_0)}{\phi^{-1}(w,\lambda)^2}.
\end{split}\]
We then use $\frac{2r}\rho\le\frac{8\chi r}{r_{\lambda,w}}$ and the definition of $r_{\lambda, w}$, \inc{p_0} of $\mathcal D$, \eqref{lambda1} and the definition of $\mathcal D$, 
\[\begin{split}
\Big(\frac{2r}{\rho}\Big)^{n+2}\frac{\lambda \mathcal{D}^{-}(\lambda_0)}{\phi^{-1}(w,\lambda)^2}
& \leq 
\bigg(\frac{8\chi}{\nu_2-\nu_1}\bigg)^{n+2} 
\max\left\{1,\bigg(\frac{\lambda}{\phi^{-1}(w,\lambda)^2}\bigg)^{-\frac{n+2}{2}}\right\} \frac{\lambda \mathcal{D}(w,\lambda_0)}{\phi^{-1}(w,\lambda)^2} \\
&\leq \max\left\{1,\bigg(\frac{\lambda}{\phi^{-1}(w,\lambda)^2}\bigg)^{-\frac{n+2}{2}}\right\} \frac{\lambda \mathcal{D}(w,\lambda)}{\phi^{-1}(w,\lambda)^2} \\
&=\min\left\{1,\left(\frac{\lambda}{\phi^{-1}(w,\lambda)^2}\right)^{\frac{n+2}{2}}\right\} \max\left\{1,\bigg(\frac{\lambda}{\phi^{-1}(w,\lambda)^2}\bigg)^{-\frac{n+2}{2}}\right\} \lambda = \lambda .
\end{split}\]
We have shown that 
\[
\fint_{Q^\lambda_{\rho}(w)}[\phi(z, |\nabla u|)+ 1]\,dz 
\le 
\lambda \quad \text{for all }\ \rho\in[\tfrac{r_{\lambda,w}}{4\chi},r_{\lambda,w}). 
\]
On the other hand, by the parabolic Lebesgue differentiation theorem from 
Section~\ref{sect:preliminaries}, we see that, for a.e. $w\in E(\nu_1,K\lambda)$, 
\[
\lim_{\rho\to0^+}\fint_{Q^\lambda_{\rho}(w)}[\phi(z, |\nabla u|)+ 1]\,dz = \phi(w, |\nabla u(w)|)+1> K\lambda.
\]
Therefore, since the mapping $\rho\mapsto \fint_{Q^\lambda_{\rho}(w)}\phi(z,|\nabla u|)\,dz$ is continuous, one can find $\rho_w\in (0,\frac{r_{\lambda,w}}{4\chi})$ such that
\[
\fint_{Q^\lambda_{\rho_w}(w)}[\phi(z, |\nabla u|)+ 1]\,dz= K\lambda
\]
and
\[
\fint_{Q^\lambda_\rho(w)}[\phi(z, |\nabla u|)+ 1]\,dz\le K\lambda
\quad \text{for all }\ \rho\in(\rho_w,r_{\lambda,w}].
\]
Since $K=4^{n+2}$, the energy normalization gives
\[
\lambda\,\big|Q^\lambda_{4\rho_w}(w)\big|
=K\lambda\,\big|Q^\lambda_{\rho_w}(w)\big|=\int_{Q^\lambda_{\rho_w}(w)}[\phi(z,|\nabla u|)+1]\,dz
\le 1.
\]
Thus the cylinders $Q^\lambda_{4\rho_w}(w)$ satisfy the volume
condition in Lemma~\ref{lem:vitali}.

Consequently, applying the Vitali-type covering (Lemma~\ref{lem:vitali}) for $\{Q^\lambda_{4\rho_w}(w)\}$, where $w\in E(\nu_1,K\lambda)$ is any parabolic Lebesgue point, we 
see that there exist $w_i\in E(\nu_1,K\lambda)$ and $\rho_i\in(0,\frac{r_{\lambda,w_i}}{4\chi})$, $i=1,2,3,\dots$, such that $Q^\lambda_{4\rho_i}(w_i)\subset Q_{\nu_2 r}$ are mutually disjoint,
\[
E(\nu_1,K\lambda)\setminus \mathcal{N}\ \subset\ \bigcup_{i=1}^\infty \chi Q^\lambda_{4\rho_i}(w_i) 
\]
for a Lebesgue measure zero set $\mathcal N$,
\[{
1\le \lambda \le K\lambda = \fint_{Q^\lambda_{\rho_i}(w_i)}[\phi(z, |\nabla u|)+ 1]\,dz \le |Q^\lambda_{\rho_i}(w_i)|^{-1}}\]
\text{and}
\[
\fint_{Q^\lambda_{\rho}(w_i)}[\phi(z, |\nabla u|)+ 1]\,dz \le K\lambda \quad \text{for all }\ \rho\in(\rho_i,r_{\lambda,w_i}).
\]


\textit{Step 3 (Estimates on super-level sets).} 
By the conclusion of Step 2, we can apply Lemma~\ref{lem:reverse} to $Q^\lambda_{4 \rho_i}(w_i)$, so that, choosing sufficiently small $\delta=\delta(n,N,p,q,\Lambda,L)\in (0,1)$, we have
\[\begin{split}
K\lambda
&= \fint_{Q^\lambda_{\rho_i}(w_i)}[\phi(z, |\nabla u|)+ 1]\,dz \leq c \bigg(\fint_{Q^\lambda_{4\rho_i}(w_i)}[\phi(z,|\nabla u|)+1]^{\theta}\,dz\bigg)^{\frac{1}{\theta}}\\
& \leq c\delta \lambda +c\bigg(\frac{1}{|Q^\lambda_{2\rho_i}|}\int_{Q^\lambda_{4\rho_i}(w_i)\cap E(\nu_2,\delta\lambda)}[\phi(z,|\nabla u|)+1]^{\theta}\,dz\bigg)^{\frac{1}{\theta}}\\
& \leq \frac12 \lambda + \frac{c}{|Q^\lambda_{4\rho_i}|}\int_{Q^\lambda_{4\rho_i}(w_i)\cap E(\nu_2,\delta\lambda)}[\phi(z,|\nabla u|)+1]^{\theta}\,dz \,\bigg(\fint_{Q^\lambda_{4\rho_i}}[\phi(z, |\nabla u|)+ 1]\,dz\bigg)^{1-\theta}\\
& \leq \frac12 \lambda+c\frac{\lambda^{1-\theta}}{|Q^\lambda_{4\rho_i}|}\int_{Q^\lambda_{4\rho_i}(w_i)\cap E(\nu_2,\delta\lambda)}[\phi(z,|\nabla u|)+1]^{\theta}\,dz .
\end{split}\]
Then we absorb $\frac12\lambda$ into the left-hand side. 
Using the property from Step 2 again, we have 
$$
\int_{\chi Q^\lambda_{4\rho_i}(w_i)}[\phi(z, |\nabla u|)+ 1]\,dz 
\le |\chi Q^\lambda_{4\rho_i}| K\lambda
\leq c \lambda^{1-\theta} \int_{Q^\lambda_{4\rho_i}(w_i)\cap E(\nu_2,\delta\lambda)}[\phi(z,|\nabla u|)+1]^{\theta}\,dz \,.
$$
Therefore, since $\{\chi Q^\lambda_{4\rho_i}(w_i)\}$ is a covering of $E(\nu_1,K\lambda)$ 
by Step 2 and the cylinders $Q^\lambda_{4\rho_i}(w_i)$ are mutually disjoint, we conclude that
\[\begin{split}
\int_{E(\nu_1,K\lambda)}[\phi(z, |\nabla u|)+ 1]\,dz 
&\leq \sum_{i=1}^\infty \int_{\chi Q^\lambda_{4\rho_i}(w_i)}[\phi(z, |\nabla u|)+ 1]\,dz\\
&\leq c \lambda^{1-\theta} \sum_{i=1}^\infty \int_{Q^\lambda_{4\rho_i}(w_i)\cap E(\nu_2,\delta\lambda)}[\phi(z,|\nabla u|)+1]^{\theta}\,dz \\
&\leq c \lambda^{1-\theta}\int_{E(\nu_2,\delta\lambda)}[\phi(z,|\nabla u|)+1]^{\theta}\,dz.
\end{split}\]
Furthermore, 
\[
\int_{E(\nu_1,\delta\lambda)\setminus E(\nu_1,K\lambda)}[\phi(z, |\nabla u|)+ 1]\,dz 
\leq (K\lambda)^{1-\theta} \int_{E(\nu_2,\delta\lambda)} [\phi(z,|\nabla u|)+1]^{\theta}\,dz. 
\]
Combining these and changing variables from $\delta\lambda$ to $\lambda$, 
we have 
\[ 
\int_{E(\nu_1,\lambda)}[\phi(z, |\nabla u|)+ 1]\,dz
\le 
c\lambda^{1-\theta} \int_{E(\nu_2,\lambda)}[\phi(z,|\nabla u|)+1]^{\theta}\,dz\ \ \ \text{for all }\lambda>\delta\lambda_1.
\]

\textit{Step 4 (Higher integrability).} 
The remaining part of the proof follows \cite{HasO21} exactly. 
For the reader's convenience, we repeat the short argument here. 
Set $\Phi(z):=\phi(z,|\nabla u(z)|)+1$, and for $k\ge 1$ and $1\le \nu \le 2$,
\[
\Phi_k(z):=\min\{\Phi(z),k\}
\quad\text{and}\quad
 E_k(\nu,\lambda):=\{z\in Q_{\nu r}:\Phi_k(z)>\lambda\}.
\]
First suppose that $k>\lambda_1$.
Since $\Phi_k\le \Phi$, it follows that $E_k\subset E$.
Then we have from Step 3 that for $\epsilon>0$, which will be determined later,
\[\begin{split}
I&:=\int_{\lambda_1}^\infty \lambda^{\epsilon-1}\int_{E_k(\nu_1,\lambda)}\Phi_k^{1-\theta} \Phi^{\theta}\,dz\,d\lambda \\
&\leq \int_{\lambda_1}^\infty \lambda^{\epsilon-1} \int_{E_k(\nu_1,\lambda)} \Phi\,dz\,d\lambda \leq c\int_{\lambda_1}^\infty \int_{E_k(\nu_2,\lambda)}\lambda^{\epsilon-\theta} \Phi^{\theta}\,dz\,d\lambda =:I\!I,
\end{split}\]
where $c$ from Step 3 is independent of $\epsilon$. 
We then apply Fubini's theorem to $I$ and $I\!I$, to conclude that
\[\begin{split}
 I = \int_{E_k(\nu_1,\lambda_1)} \Phi_k^{1-\theta} \Phi^{\theta} \int_{\lambda_1}^{\Phi_k} \lambda^{\epsilon-1}\, d\lambda\, dz= \frac{1}{\epsilon} \int_{E_k(\nu_1,\lambda_1)} \left[\Phi_k^{1-\theta+\epsilon} \Phi^{\theta}-\lambda_1^{\epsilon} \Phi_k^{1-\theta}\Phi^{\theta}\right]\,dz
\end{split}\]
and
\[\begin{split}
I\!I= c\int_{E_k(\nu_2,\lambda_1)}\Phi^{\theta} \int_{\lambda_1}^{\Phi_k} \lambda^{\epsilon-\theta} \,d\lambda \,dz
&= \frac{c}{1+\epsilon-\theta}\int_{E_k(\nu_2,\lambda_1)} \left(\Phi_k^{1+\epsilon-\theta} -\lambda_1^{1+\epsilon-\theta}\right)\Phi^{\theta} \, dz\\
&\leq \frac{c}{1-\theta}\int_{E_k(\nu_2,\lambda_1)} \Phi_k^{1+\epsilon-\theta} \Phi^{\theta} \, dz.
\end{split}\] 
Therefore
 \[\begin{split}
\int_{E_k(\nu_1,\lambda_1)} \Phi_k^{1-\theta+\epsilon} \Phi^{\theta}\,dz \leq \lambda_1^{\epsilon} \int_{E_k(\nu_1,\lambda_1)}\Phi\,dz + c\epsilon \int_{Q_{\nu_2 r}}\Phi_k^{1-\theta+\epsilon} \Phi^{\theta} \, dz.
\end{split}\] 
At this stage, we choose $\epsilon=\epsilon(n,N,p,q,\Lambda,L)>0$ so small that 
$c\epsilon\leq \frac12$. On the other hand, 
\[
\int_{Q_{\nu_1r}\setminus E_k(\nu_1,\lambda_1)} \Phi_k^{1-\theta+\epsilon} \Phi^{\theta}\,dz
\le 
\lambda_1^{\epsilon}\int_{Q_{\nu_1 r}} \Phi_k^{1-\theta} \Phi^{\theta}\,dz
\le 
\lambda_1^{\epsilon}\int_{Q_{2 r}} \Phi \,dz.
\] 
Combining these two estimates and 
$\lambda_1^\epsilon\leq \frac{c}{(\nu_2-\nu_1)^{\alpha_0}}\lambda_0^\epsilon$ from 
\eqref{lambda1} with $\alpha_0:=\epsilon (n+2)\max\{\frac{\tilde q}{2},\frac{2p}{p(n+2)-2n}\}$, we obtain
\[\begin{split}
\int_{Q_{\nu_1r}} \Phi_k^{1-\theta+\epsilon} \Phi^{\theta}\,dz 
\leq 
\frac12 \int_{Q_{\nu_2r}}\Phi_k^{1-\theta+\epsilon} \Phi^{\theta} \, dz + \frac{c \lambda_0^{\epsilon} }{(\nu_2-\nu_1)^{\alpha_0}}\int_{Q_{2r}} \Phi\,dz,
\end{split}\] 
for all $1\le \nu_1<\nu_2\le 2$.
If $k\le\lambda_1$, the same estimate follows directly from
$\Phi_k^{1-\theta+\epsilon}\Phi^\theta\le k^\epsilon\Phi\le\lambda_1^\epsilon\Phi$.
Thus it holds for every fixed $k\ge1$ and every pair of radii. The iterated quantity is finite because $\Phi_k^{1-\theta+\epsilon}\Phi^\theta\le k^\epsilon\Phi$.
Applying a standard iteration (see e.g. \cite[Lemma~6.1]{Giusti_book}) to this inequality, we find that 
\[\begin{split}
\int_{Q_{r}} \Phi_k^{1-\theta+\epsilon} \Phi^{\theta}\,dz 
\leq c \lambda_0^{\epsilon} \int_{Q_{2r}} \Phi\,dz,
\end{split}\]
Finally, letting $k\to \infty$ and recalling 
the definition of $\lambda_0$, we have by monotone convergence 
that
\[\begin{split}
\fint_{Q_{r}} \Phi^{1+\epsilon}\,dz & \leq c\lambda_0^{\epsilon } \fint_{Q_{2r}} \Phi\,dz 
= 
c \bigg[ (\mathcal{D}^-)^{-1} \bigg(\fint_{Q_{2r}}\Phi\,dz\bigg)\bigg]^\epsilon\fint_{Q_{2r}}\Phi\,dz. 
\end{split}\]
This implies the inequality from Theorem~\ref{thm:main}.
\end{proof}

\section*{Acknowledgement}
The authors are deeply grateful to the anonymous referees for their constructive suggestions and careful review. J. Ok was supported by the National Research Foundation of Korea (NRF) funded by the 
Korean government (MSIT) (2022R1C1C1004523).

\subsection*{Data availability} Data sharing is not applicable to this article as obviously no datasets were generated or
analyzed during the current study.

\subsection*{Use of artificial intelligence} 
AI tools were used for spell-checking and French translation but played 
no role in the mathematical content.

\subsection*{Conflict of Interest} The authors declare no conflict of interest.


\bibliographystyle{amsplain}

\end{document}